\documentclass{jcmlatex}
\def\JCMvol{xx}
\def\JCMno{x}
\def\JCMyear{202x}
\def\JCMreceived{xxx}
\def\JCMrevised{xxx}
\def\JCMaccepted{xxx}

\usepackage{amssymb}
\usepackage{amsmath}
\usepackage{xcolor}
\usepackage{cases}
\usepackage{booktabs}
\usepackage{algorithm,algorithmic}
\usepackage{rotating}
\usepackage{subfigure}
\usepackage{lipsum}
\usepackage{amsfonts,mathrsfs}
\usepackage{graphicx}
\usepackage{epstopdf}
\usepackage{algorithmic}
\usepackage{multirow,booktabs,cases,bm}
\usepackage{amsopn}

\ifpdf
\DeclareGraphicsExtensions{.eps,.pdf,.png,.jpg}
\else
\DeclareGraphicsExtensions{.eps}
\fi
\usepackage{enumitem}
\setlist[enumerate]{leftmargin=.5in}
\setlist[itemize]{leftmargin=.5in}

\def\[{\begin{equation}}
	\def\]{\end{equation}}

\newcommand{\dom}[1]{\mathrm{\bf dom}\,{(#1)}} %Domain
\newcommand{\prox}{\mathrm{\bf Prox}} %Graph
\newcommand{\dist}{\mathrm{\bf dist}} %Distance between point and set

\def\nn{\nonumber}

\def\E{{\mathcal E}}
\def\u{{\bm{u}}}
\def\t{\top}

\def\B{\mathscr{B}}
\def\R{{\mathbb R}}

\def\L{{\mathscr L}}

\def\X{{\mathcal X}}
\def\Y{{\mathcal Y}}

\def\U{{\mathcal U}}

\def\bbf{{\bm{f}}}

\def\bx{{\bm x}}
\def\by{{\bm{y}}}

\def\bA{{\bm{A}}}

\begin{document}
\graphicspath{{./FIG/},{./PIC/}}

\markboth{H. HE, K. WANG AND J. YU}{A symmetric primal-dual algorithmic framework for saddle point problems}

\title{A SYMMETRIC PRIMAL-DUAL ALGORITHMIC FRAMEWORK FOR SADDLE POINT PROBLEMS}

\author{
	Hongjin He\thanks{School of Mathematics and Statistics, Ningbo University, Ningbo, China. \\ Email: hehongjin@nbu.edu.cn}
\and
    Kai Wang\thanks{School of Mathematics and Statistics, Nanjing University of Science and Technology, Nanjing, China.\\ Email: wangkaihawk@njust.edu.cn}
\and
    Jintao Yu \thanks{School of Mathematics and Statistics, Ningbo University, Ningbo, China. \\ Email: yujintao0045@163.com}
    }

\maketitle

\begin{abstract}
In this paper, we propose a new primal-dual algorithmic framework for a class of convex-concave saddle point problems frequently arising from image processing and machine learning. Our algorithmic framework updates the primal variable between the twice calculations of the dual variable, thereby appearing a symmetric iterative scheme, which is accordingly called the {\bf s}ymmetric {\bf p}r{\bf i}mal-{\bf d}ual {\bf a}lgorithm (SPIDA). It is noteworthy that the subproblems of our SPIDA are equipped with Bregman proximal regularization terms, which make SPIDA versatile in the sense that it enjoys an algorithmic framework to understand the iterative schemes of some existing algorithms, such as the classical augmented Lagrangian method (ALM), linearized ALM, and Jacobian splitting algorithms for linearly constrained optimization problems. Besides, our algorithmic framework allows us to derive some customized versions so that SPIDA works as efficiently as possible for structured optimization problems. Theoretically, under some mild conditions, we prove the global convergence of SPIDA and estimate the linear convergence rate under a generalized error bound condition defined by Bregman distance. Finally, a series of numerical experiments on the basis pursuit, robust principal component analysis, and image restoration demonstrate that our SPIDA works well on synthetic and real-world datasets.
\end{abstract}



\section{Introduction}\label{sec1}
Recently, saddle point problems have received considerable attention in the signal/image processing, machine learning, and optimization communities, e.g., see \cite{CP11,CP16,EZC10,RHLNSH20,ZC08b}, to name just a few.
In this paper, we are interested in the convex-concave saddle point problem with a bilinear coupling term, which takes the following form:
\begin{equation}\label{Problem}
	\min_{x \in \X}\max_{y \in \Y} \left\{\L(x,y) := f(x) +\langle Ax, y \rangle  - g(y)\right\},
\end{equation}	
where $\X\subseteq \R^n$ and $\Y\subseteq \R^m$ are two closed nonempty convex sets, both $f(\cdot): \X \rightarrow (-\infty, \infty]$ and $g(\cdot): \Y \rightarrow (-\infty, \infty]$ are proper closed convex (possibly nonsmooth) functions, $\langle \cdot, \cdot \rangle$ represents the standard inner product of vectors, and $A: \R^n \rightarrow \R^m$ is a bounded linear operator. It is interesting that \eqref{Problem} provides a unified framework for the treatment of convex composite optimization problem
\begin{equation*}\label{P_Problem}
	\min_{x \in \R^n} \left\{f(x)  + g^*(Ax)\right\}
\end{equation*}
and the canonical convex minimization problem with linear constraints (see Section \ref{Sec-LCM}), where $g^*(\cdot)$ is the Fenchel conjugate of function $g(\cdot)$. 

To efficiently exploit the min-max structure of \eqref{Problem}, a seminal work can be traced back to the Arrow-Hurwicz Primal-Dual (AHPD) method \cite{AHU58}, which updates the primal and dual variables in a sequential order by solving two optimization subproblems as follows:
\begin{equation}\label{AHPD}
	\left\{
	\begin{aligned}
		x^{k+1} &= \arg\min_{x\in \mathcal{X}}\left\{ f(x) + \langle Ax, y^{k}\rangle + \frac{\mu}{2} \| x-x^{k} \|^{2} \right\}, \\
		y^{k+1} &= \arg\max_{y\in \mathcal{Y}}\left\{ -g(y) + \langle A x^{k+1}, y\rangle - \frac{\gamma}{2} \| y-y^{k} \|^{2} \right\},
	\end{aligned}\right.
\end{equation}
where $\mu$ and $\gamma$ are two positive proximal regularization parameters serving as step sizes for updating. In the literature, such a method is also reemphasized as primal-dual hybrid gradient (PDHG) method with fruitful applications in image processing \cite{BR12,EZC10,HMY17,ZC08b}. Although some convergence properties have been established under additional conditions \cite{EZC10,HYY15,MSMC15}, the most recent work \cite{HXY22} showed that the AHPD method with any constant step size is not necessarily convergent for solving generic convex-concave saddle point problems.
In 2011, Chambolle and Pock \cite{CP11} judiciously introduced a first-order primal-dual algorithm by absorbing an extrapolation step for algorithmic acceleration. Note that such an algorithm is commonly denoted by PDHG in the optimization literature. Therefore, we also use PDHG to represent the first-order primal-dual algorithm \cite{CP11} throughout this paper. For given the $k$-th iterate $(x^k,y^k)$, the iterative scheme of the PDHG \cite{CP11} reads as
\begin{numcases}{\label{f1}}
	x^{k+1} = \arg\min_{x\in \mathcal{X}}\left\{ f(x) + \langle Ax, y^{k}\rangle + \frac{\mu}{2} \| x-x^{k} \|^{2} \right\}, \nonumber\\
	\tilde{x}^{k+1} = x^{k+1} + \tau(x^{k+1}-x^{k}), \label{fopda} \\
	y^{k+1} = \arg\max_{y\in \mathcal{Y}}\left\{-g(y) + \langle A \tilde{x}^{k+1}, y\rangle - \frac{\gamma}{2} \| y-y^{k} \|^{2} \right\}, \nonumber
\end{numcases}
where $\tau \in [0, 1]$ is an extrapolation parameter and both $\mu > 0$ and $\gamma > 0$ are regularization parameters. In particular, they further proved some convergence properties for the special case $\tau=1$ under the following requirement:
\begin{equation}\label{condition}
	\|A A^\top\| < \mu\gamma.
\end{equation}
It is notable that the PDHG not only requires weaker convergence-guaranteeing conditions, but also runs faster than the AHPD method (see \cite{CP11,CP16b}). In recent years, there are some papers contributed to further studies on the extrapolation step, e.g., see \cite{CHX13,CY21,HY12b,HDW16,WH20}. As aforementioned,  both $\mu$ and $\gamma$ serve as step sizes for updating. From computational perspective, larger step sizes usually lead to faster convergence, which accordingly encourages researchers to relax condition \eqref{condition}, e.g., see \cite{HMXY22,JZH23,LY24,LY21}. When the $x$- and $y$-subproblems are not easily to be solved or the maximum eigenvalue of $A A^\top$ cannot be easily evaluated in some cases, a better way is to solve the underlying subproblems in an inexact way or the employment of line search for avoiding the calculation of $\|A A^\top\|$, e.g., see \cite{CYZ22,JCWH21,JWCZ21,MP18,RC20}. In recent years, saddle point problems and primal-dual algorithms received much attention in the fields of machine learning, imaging science, and optimization. Here, we only refer the reader to \cite{CKCH23,KP15,RHLNSH20,Val21} for recent surveys and references therein along this direction.

It is well-known that the saddle point problem \eqref{Problem} provides a powerful treatment for linearly constrained convex optimization problems (e.g., see Section \ref{Sec-LCM}). In this application, the dual variable $y$ in \eqref{Problem} serves as the so-called Lagrangian multiplier. However, as shown in the excellent overview on a phenomenon of slow convergence of optimization algorithms \cite{IS15} (also see some comments \cite{Fis15,Mar15,Mor15,Rob15}), critical multipliers are nonempty for optimal solutions with nonunique Lagrangian multipliers, which play a crucial negative role in numerical optimization yielding slow convergence of major primal-dual algorithms, including Newton and Newton-related methods, the Augmented Lagrangian Method (ALM), and the sequential quadratic programming method. Therefore, these surprising discoveries clearly demonstrate that the updating scheme of the dual variable (i.e., Lagrangian multiplier) is very important for algorithmic acceleration, which motivates us to develop some ``dual stabilization techniques'' for solving \eqref{Problem}. Besides, most of saddle point problems arising from machine learning and image processing display unbalanced primal and dual subproblems in the sense that the dual problem is often easier than the primal one (e.g., see \cite{CP11,CP16} and also Section \ref{Sec-LCM}). As studied in the most recent work \cite{HY21,MCJH23}, balancing the subproblems of the classical ALM is able to greatly speed up the convergence of solving linearly constrained optimization problems. Therefore, how to balance both subproblems and what will be produced by some balancing technique for \eqref{Problem} are also motivations of this paper.

Considering the different complexity of primal and dual subproblems, we in this paper employ the symmetric spirit to design a new primal-dual algorithmic framework for saddle point problem \eqref{Problem}, where the primal variable is updated once between the twice calculations of the dual variable. To a certain extent, our algorithm is able to balance the computation of primal and dual subproblems, and the one more calculation of the dual variable can be regarded as some dual stabilization technique from numerical perspective.
Since the proposed algorithm appears a symmetric updating order on the dual variable, we call it symmetric primal-dual algorithm and denote it by SPIDA for simplicity.
A toy example shows that our SPIDA has a nice convergence behavior, while the AHPD fails to converge and the PDHG runs a little slower than our SPIDA for some proximal parameters.
Notice that each subproblem is equipped with a Bregman proximal term to make our algorithm versatile so that we can easily derive the iterative schemes of some classical first-order optimization methods, including the ALM and its linearized version for one-block linearly constrained convex optimization problems, and some Jacobian splitting algorithms for multi-block linearly constrained convex minimization problems.
Particularly, our algorithmic framework is of benefit for producing some customized variants for linearly constrained optimization problems. Theoretically, we prove that our SPIDA is globally convergent under standard conditions, while the linear convergence rate is also estimated under a generalized error bound condition defined by Bregman distance.
Finally, a series of numerical experiments on basis pursuit, robust principal component analysis (RPCA), and image restoration demonstrate that our SPIDA performs better than some state-of-the-art primal-dual algorithms in many cases.

The remainder of paper is organized as follows. In Section \ref{Sec2}, we recall some notations and definitions that will be used throughout this paper. In Section \ref{Sec3}, we first describe the details of our proposed SPIDA for \eqref{Problem}. Then, we prove its global convergence and estimate its linear convergence rate. In Section \ref{Sec-LCM}, we apply our SPIDA to solve the linearly constrained convex minimization problems, and show that some classical first-order optimization methods are special cases of the proposed SPIDA. In Section \ref{Sec5}, we conduct the numerical performance of our SPIDA on solving some structured optimization problems with synthetic and real datasets. Finally, we complete this paper with drawing some conclusions in Section \ref{Sec6}.

\section{Preliminaries}\label{Sec2}
In this section, we summarize some notations and basic concepts that will be used in subsequent analysis.

Throughout this paper, the superscript symbol $^\top$ represents the transpose for vectors and matrices. Let $\R^n$ be an $n$-dimensional Euclidean space endowed with the $M$-inner product $\langle x,y\rangle_M=\langle x, My\rangle=x^\top M y$, where $x,y\in \R^n$ and $M$ is a symmetric and positive definite (or semi-definite) matrix ($M\succ 0$ (or $\succeq 0$) for short). Consequently, for a given vector $x\in \R^n$, we define the $M$-norm by
\begin{equation*}\label{Mnorm}
	\|x\|_{M}=\sqrt{\left\langle x, Mx\right\rangle}.
\end{equation*}
In particular, when $M$ is an identity matrix, the $M$-norm reduces to the standard Euclidean norm (denoted by $\|x\|$). Moreover, for a matrix $A$, we use $\|A\|$ to represent the square root of the maximum eigenvalue of $A^\top A$.

\begin{definition}\label{def:lsc}
	Let $f(\cdot):\R^n \to  [-\infty,+\infty] $ be an extended real-valued function, and denote the domain of $f(\cdot)$ by
	$\dom{f} := \left\{x\in \R^n\;|\;f(x) <\infty\right\}$. Then, we say that the function $f(\cdot)$ is
	\begin{enumerate}			
		\item[\rm (i)] proper if $f(x)  >-\infty$ for all $x \in \R^n$ and $\dom{f}\neq \emptyset$;
		\item[\rm (ii)] convex if $f\left(tx+(1-t)y\right)\leq t f(x) +(1-t)f(y)$ for any $x,y\in \dom{f}$ and $t\in [0,1]$;
		\item[\rm (iii)] $\varrho$-strongly convex with a given $\varrho>0$ if $\dom{f}$ is convex and the following inequality holds for any $x,y\in \dom{f}$ and $t\in[0,1]$:
		$$f(tx + (1-t)y)\leq t f(x) + (1-t)f(y) - \frac{\varrho}{2}t(1-t)\|x-y\|^2.$$
	\end{enumerate}
\end{definition}

Let $f(\cdot):\R^n\to(-\infty,+\infty]$ be a proper, closed and convex function, then the subdifferential of $f(\cdot)$ at $x \in \dom{f}$ is given by
$$\partial f(x)= \left\{\; \xi\,|\, f(z) \geq f(x) + \left\langle z-x, \xi\right\rangle, \,\forall \, z \in \dom{f}  \;\right\}.$$
In what follows, we denote $\dom{\partial f}:=\{x\in\R^n\;|\;\partial f(x)\neq \emptyset\}$. Then, the following first-order characterizations of strong convexity are frequently used for analysis (e.g., see \cite[Theorem 5.24]{Beck17}).
\begin{lemma}\label{lem:strong}
	Let $f(\cdot): \R^n\to (-\infty,\infty]$ be a proper closed and convex function. Then, for a given $\varrho>0$, the following three claims are equivalent:
	\begin{itemize}
		\item[\rm (i)] $f(\cdot)$ is $\varrho$-strongly convex.
		\item[\rm (ii)] $f(y)\geq f(x)+\langle \xi, y-x\rangle + \frac{\varrho}{2}\|y-x\|^2$ for any $x\in\dom{\partial f}$, $y\in\dom{f}$, and $\xi\in\partial f(x)$.
		\item[\rm (iii)] $\langle \xi-\eta, x-y\rangle \geq \varrho \|x-y\|^2$ for any $x,y\in\dom{\partial f}$ and $\xi \in\partial f(x)$, $\eta\in\partial f(y)$.
	\end{itemize}
\end{lemma}

The proximal operator of $f$ (see \cite{Mor62,PB13}), denoted by $\prox_{f}(\cdot)$, is given by
\begin{equation*}
	\prox_{f}(a) = \arg\min_{x \in \R^n } \left\{f(x) + \frac{1}{2}\|x-a\|^2 \right\},\quad a\in\R^n.
\end{equation*}
Particularly, if $f(\cdot)$ is the indicator function $\delta_\X(\cdot)$ associated with the nonempty convex set $\X$, i.e.,
\begin{equation*}
	\delta_\X(x)=\left\{ \begin{array}{ll}
		0, &\;\; \text{if}\;x\in \X, \\ +\infty, &\;\; \text{otherwise},
	\end{array}\right.
\end{equation*}
then the proximal operator $\prox_{f}(\cdot)$ immediately reduces to the projection operator, i.e., $\prox_f(\cdot) \equiv \Pi_\X(\cdot)$.
Let $\Omega$ be a nonempty closed convex set of $\R^n $, we define
$$\dist_M(x,\Omega):=\min\;\left\{\|x-z\|_M\;|\;z \in \Omega\right\}$$
as the distance from any $x \in \R^n $ to the set $\Omega$ in the sense of matrix norm, where $M$ is a given symmetric and
positive definite matrix. In particular,  when $M$ is an identity matrix, we use
$\dist(x,\Omega)$ to denote the Euclidean distance from any $x$ to the set $\Omega$ for simplicity.

Given a proper closed strictly convex function $\phi(\cdot) : \R^n \to (-\infty,+\infty]$, finite at $x$, $y$ and differentiable at $y$, the Bregman distance \cite{Bre66} between $x$ and $y$ associated with the kernel function $\phi$ is defined as
\begin{equation*}
	\B_{\phi}(x,y) = \phi(x) - \phi(y) - \langle \nabla \phi(y),x-y \rangle,
\end{equation*}
where $\nabla \phi(y)$ represents the gradient of $\phi$ at point $y$. It is not difficult to see that the Bregman distance covers the standard Euclidean distance as its special case when $\phi(\cdot)=\frac{1}{2}\|\cdot\|^2$. Here, we summarize three widely used Bregman distances in Table \ref{Tab_Bregman}. However, the Bregman distance does not always share the symmetry and the triangle inequality property with the Euclidean distance.
\begin{table}[htbp]
	\caption{Three popular Bregman distances.}\label{Tab_Bregman}
	\centering
	\renewcommand\arraystretch{1.5}
	\begin{tabular*}{\textwidth}{@{\extracolsep{\fill}}cll}	\toprule
		Type & Kernel function $\phi(\cdot)$ & Bregman distance $\B_\phi(x,y)$ \\ \midrule
		I & $\frac{1}{2}\|z\|^2$ & $\frac{1}{2}\|x-y\|^2$, \; $\forall x,y\in\R^n$  \\
		II & $\frac{1}{2}\|z\|_M^2$ with $M\succ 0$& $\frac{1}{2}\|x-y\|_M^2$, \; $\forall x,y\in\R^n$ \\
		III & $\sum_{i=1}^n z_i \log z_i$ & $\sum_{i=1}^n x_i \log \frac{x_i}{y_i} + y_i-x_i, \quad \forall x\in\R^n_+,\;\forall y\in\R^n_{++}$\\ \bottomrule
	\end{tabular*}	
\end{table}

Below, we summarize some properties of Bregman distance \cite{Beck17}.

\begin{lemma}\label{lem:Bregman}
	Suppose that $\Omega\subseteq \R^n$ is nonempty closed and convex, and the function $\phi(\cdot)$ is proper closed convex and differentiable over $\dom{\partial \phi}$. If $\Omega\subseteq \dom{\phi}$ and $\phi(\cdot)+\delta_{\Omega}(\cdot)$ is $\varrho$-strongly convex ($\varrho>0$), then the Bregman distance $\B_\phi(\cdot,\cdot)$ associated with $\phi(\cdot)$ has the following properties:
	\begin{itemize}
		\item[\rm (i)] $\B_{\phi}(x,y)\geq \frac{\varrho}{2}\|x-y\|^2$ for all $x\in\Omega$ and $y\in\Omega\cap \dom{\partial \phi}$;
		\item[\rm (ii)] Let $x\in\Omega$ and $y\in\Omega\cap \dom{\partial \phi}$. Then $\B_{\phi}(x,y)\geq 0$, and in particular, the equality holds if and only if $x=y$;
		\item[\rm (iii)] For $a,b\in \dom{\partial \phi}$ and $c\in\dom{\phi}$, the following equality holds:
		\begin{equation*}\label{three-point}
			\langle \nabla\phi(b)-\nabla\phi(a), c-a\rangle =\B_{\phi}(c,a)+\B_{\phi}(a,b)-\B_{\phi}(c,b).
		\end{equation*}
	\end{itemize}
\end{lemma}

Below, we present the first-order optimality condition of \eqref{Problem}. The pair $(x^\star, y^\star)$ defined on $\X  \times \Y $ is called a saddle point of \eqref{Problem} if it satisfies the following inequalities
\begin{equation*}
	\L (x^\star,y) \leq \L(x^\star,y^\star) \leq \L(x,y^\star), \quad \forall\, x \in \X , \; \forall\, y \in \Y,
\end{equation*}
which can be further reformulated as a mixed variational inequality:
\begin{equation}\label{KKT}
	\left\{\begin{aligned}
		f(x)-f(x^\star) + \left\langle x-x^\star, A^\top y^\star \right\rangle &\geq 0,\quad \forall x\in \X, \\
		g(y)-g(y^\star) + \left\langle y-y^\star, -Ax^\star \right\rangle &\geq 0,\quad \forall y\in \Y,
	\end{aligned}\right.
\end{equation}
or equivalently,
\begin{subequations}\label{MVI}
	\begin{equation}\label{MVI-a}
		\varUpsilon(\u)-\varUpsilon(\u^\star) + \langle \u-\u^\star, G\u^\star\rangle \geq 0,\quad \forall \u\in\U,
	\end{equation}
	where
	\begin{equation}\label{MVI-b}
		\u=\left(\begin{array}{c}x \\ y\end{array}\right),\;\; \varUpsilon(\u)=f(x)+g(y),\;\; G=\left(\begin{array}{cc} 0 & A^\top \\ -A & 0\end{array}\right),\;\;\U=\X\times\Y.
	\end{equation}
\end{subequations}
Alternatively, it is well-known (see \cite{YH16}) that solving \eqref{MVI} amounts to finding a solution of a generalized projection equation, which is shown by the following lemma.

\begin{lemma}\label{vi-proj}
	The variational inequality problem \eqref{MVI} amounts to finding $\u^{\star}:=(x^\star,y^\star)$ such that $0\in \E(\u^\star, t)$, i.e.,
	$$\dist^2(0, \E(\u^\star,t))=0,$$
	where the set-valued mapping $\E(\u,t)$  is defined as
	\begin{equation}\label{error25}
		\E(\u,t):=\left(\begin{array}{ll}
			\E_{\X}(\u,t):=x-\Pi_{\X}\left[x-t(\partial f(x)+A^\t y ) \right]
			\\ \E_{\Y}(\u,t):=y-\Pi_{\Y}\left[y-t (\partial g(y)-Ax ) \right]
		\end{array} \right)
	\end{equation}
	with $t>0$ being an arbitrary scalar. 
\end{lemma}

Throughout this paper, we let $\U^\star$ be the solution set of \eqref{MVI}, which is assumed to be nonempty. Clearly,
it follows from Lemma \ref{vi-proj} that
$$\U^\star=\{\u^\star ~|~ \dist(0,\E(\u^\star,t))=0\}.$$
Notice that our convergence rate analysis under the error bound condition is based on the variational inequality characterization \eqref{KKT} and the related theory of variational inequalities.

\section{Algorithm and Convergence Properties}\label{Sec3}
In this section, we first present the algorithmic framework for \eqref{Problem} and show that our algorithm has a nice convergence behavior through a toy example. Then, we prove that our algorithm is globally convergent and has a linear convergence rate under some standard conditions (see \eqref{condition} or more details in Remark \ref{rm1}) used in \cite{CP11}.

\subsection{Algorithmic framework}
Considering the possibly unbalanced complexity of primal and dual subproblems in some cases, we are motivated to update the primal variable between the twice calculations of the dual variable so that updating the dual variable in a symmetric way. The extra calculation of the dual variable accordingly balances the speed of both variables approaching to their optimal solutions, which can also be regarded as one stabilization on the dual variable.  Our algorithmic framework is described formally in Algorithm \ref{alg1}.

\begin{algorithm}[!htbp]
	\caption{The Symmetric Primal-Dual Algorithmic Framework for \eqref{Problem}.}\label{alg1}
	\begin{algorithmic}[1]
		\STATE Choose starting points $x^{0} \in \X,y^{0} \in \Y$ and parameters $\gamma>0$ and $\mu>0$.
		\REPEAT
		\STATE Update $x$ and $y$ via
		\begin{align}
			&\tilde{y}^{k+1} = \arg\max_{y\in \mathcal{Y}}\left\{ -g(y) + \langle Ax^k, y\rangle - \gamma \B_{\phi}(y,y^{k}) \right\}, \label{SPIDA-a}\\
			&x^{k+1} = \arg\min_{x\in \mathcal{X}}\left\{ f(x) + \langle Ax, \tilde{y}^{k+1} \rangle + \mu \B_{\psi}(x,x^{k}) \right\} \label{SPIDA-b},  \\
			&y^{k+1} = \arg\max_{y\in \mathcal{Y}}\left\{ -g(y) + \langle Ax^{k+1}, y\rangle - \gamma  \B_{\phi}(y,y^{k}) \right\}, \label{SPIDA-c}
		\end{align}
		\UNTIL some stopping criterion is satisfied.
		\RETURN an approximate saddle point $(\hat{x},\hat{y})$.
	\end{algorithmic}
\end{algorithm}

\begin{remark}\label{remark-one}
	It is noteworthy that the embedded Bregman proximal regularization terms (i.e., $\B_{\psi}(x,x^k)$ and $\B_{\phi}(y,y^k)$) make our Algorithm \ref{alg1} versatile in the sense that we can choose appropriate Bregman kernel functions as listed in Table \ref{Tab_Bregman} to interpret some state-of-the-art first-order optimization solvers (see Section \ref{Sec-LCM}), or design customized variants for some real-world problems (see Section \ref{Sec5}).
\end{remark}

\begin{figure}[!htb]
	\centering
	\subfigure[AHPD ($\mu=\gamma=1$)]{
		\begin{minipage}[t]{0.32\linewidth}
			\centering
			\includegraphics[width=1\textwidth]{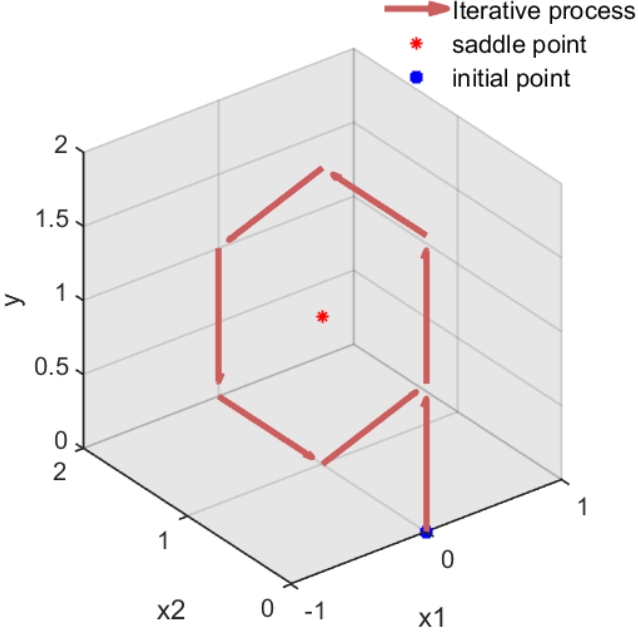}
			%\caption{fig1}
		\end{minipage}%
	}%
	\subfigure[PDHG ($\mu=\gamma=1$)]{
		\begin{minipage}[t]{0.32\linewidth}
			\centering
			\includegraphics[width=1\textwidth]{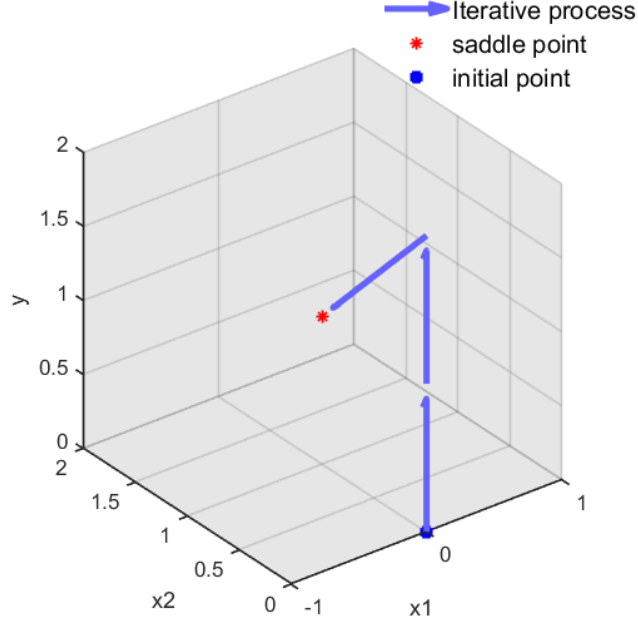}
			%\caption{fig2}
		\end{minipage}%
	}%
	\subfigure[Algorithm \ref{alg1} ($\mu=\gamma=1$)]{
		\begin{minipage}[t]{0.32\linewidth}
			\centering
			\includegraphics[width=1\textwidth]{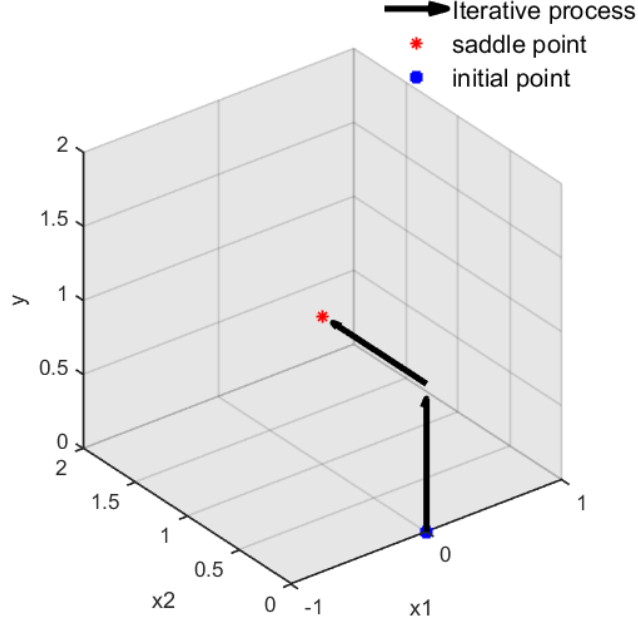}
		\end{minipage}
	}%
	
	\subfigure[AHPD ($\mu=\gamma=\sqrt{2}$)]{
		\begin{minipage}[t]{0.32\linewidth}
			\centering
			\includegraphics[width=1\textwidth]{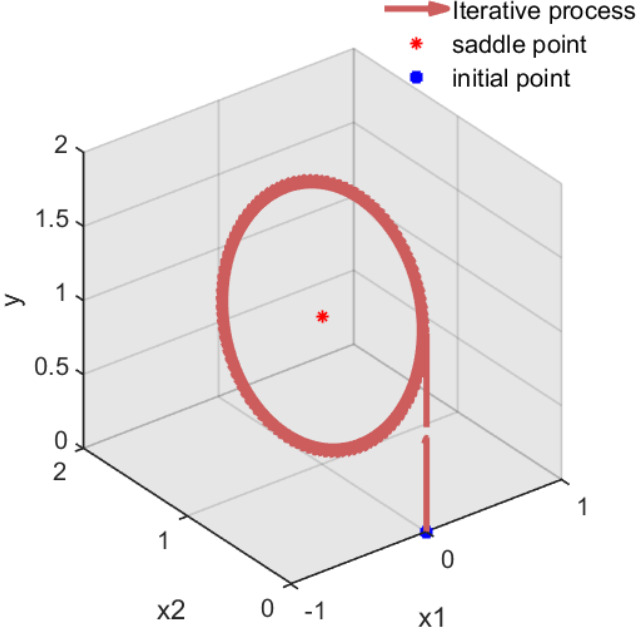}
		\end{minipage}%
	}%
	\subfigure[PDHG ($\mu=\gamma=\sqrt{2}$)]{
		\begin{minipage}[t]{0.32\linewidth}
			\centering
			\includegraphics[width=1\textwidth]{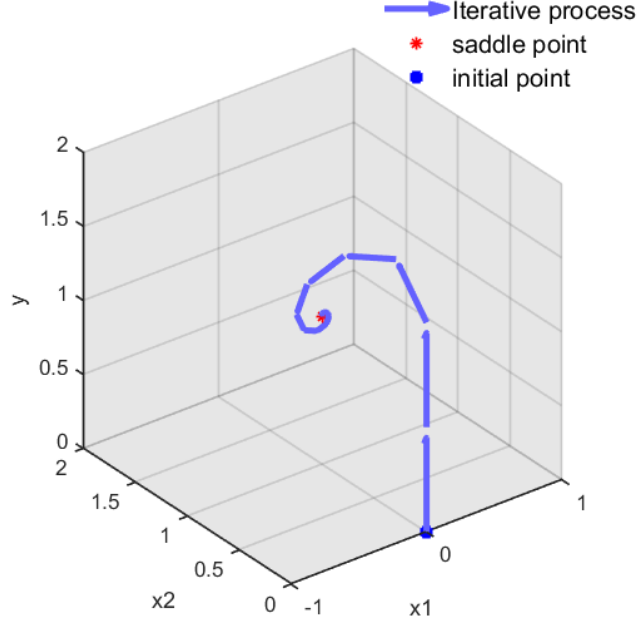}
		\end{minipage}%
	}%
	\subfigure[Algorithm \ref{alg1} ($\mu=\gamma=\sqrt{2}$)]{
		\begin{minipage}[t]{0.32\linewidth}
			\centering
			\includegraphics[width=1\textwidth]{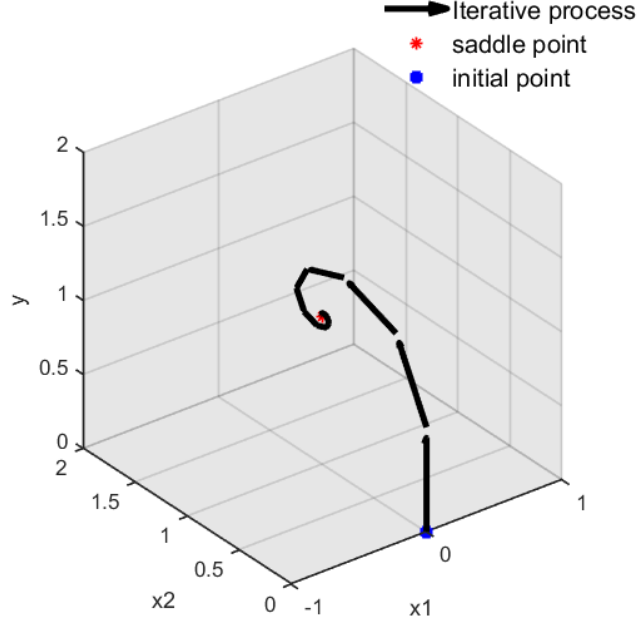}
		\end{minipage}
	}%
	\caption{Illustration of the convergence behaviors of AHPD, PDHG and Algorithm \ref{alg1} for the toy example \eqref{Problem2} with setting different parameters $\mu$ and $\gamma$.}
	\label{figure}
\end{figure}

\begin{remark}\label{remark-two}
	We employ a toy example used in \cite{HXY22} to show that our Algorithm \ref{alg1} enjoys a nice convergence behavior. Consider the following linear programming:
	\begin{equation}\label{Problem2}
		\min_{x_{1},x_{2}} \left\{\;  2x_{1} + x_{2} \;|\; x_{1} +x_{2} =1,\; x_{1} \ge 0 , \; x_{2} \ge 0\; \right\},
	\end{equation}
	which has a unique solution $(x_{1}^{*},x_{2}^{*})=(0,1)$. Moreover, the dual problem of \eqref{Problem2} is
	\begin{equation*}
		\max_{y} \left\{\;  y\;|\; y \le 1,\; y \le 2\; \right\}
	\end{equation*}
	and its optimal solution is $y^{*}=1$. Accordingly, we reformulate \eqref{Problem2} as a standard form of saddle point problems, i.e.,
	\begin{equation}\label{Linear-saddle}
		\min_{x_{1}\geq 0,x_{2}\geq 0} \max_y\left\{\;  2x_{1} + x_{2} -y\left( x_{1} +x_{2} -1\right)\; \right\}.
	\end{equation}
	Applying Algorithm \ref{alg1} to \eqref{Linear-saddle} by setting the Bregman kernel functions as the first type listed in Table \ref{Tab_Bregman}, the iterative scheme is immediately specified as
	\begin{subnumcases}{\label{f5}}
		\tilde{y}^{k+1} =y^{k} - \frac{1}{\gamma}\left( x^{k}_{1} + x^{k}_{2} -1   \right), \nonumber \\
		x^{k+1}_{1} = \max \left\{ \left(- \frac{2}{\mu} + \frac{1}{\mu}\tilde{y}^{k+1} +x_{1}^{k}  \right),0  \right\},  \nonumber \\
		x^{k+1}_{2} = \max \left\{ \left(- \frac{1}{\mu} + \frac{1}{\mu}\tilde{y}^{k+1} +x_{2}^{k}  \right),0  \right\},   \nonumber \\
		y^{k+1} = y^{k} - \frac{1}{\gamma}\left( x^{k+1}_{1} + x^{k+1}_{2} -1   \right). \nonumber
	\end{subnumcases}
	Also, we implement the AHPD method \eqref{AHPD} and the PDHG \eqref{fopda} with $\tau=1$ to solve \eqref{Linear-saddle}. Here, we take $(x^{0}_{1} , x^{0}_{2} , y^{0}) = (0,0,0)$ as starting points and plot trajectories of the sequences generated by the three algorithms in Fig. \ref{figure} for the cases where $\mu= \gamma = 1 $ or $\mu=\gamma = \sqrt{2}$, respectively. It can be easily seen from Fig. \ref{figure} that the sequence generated by our Algorithm \ref{alg1} converges slightly faster than PDHG to the unique optimal solution, while the AHPD method generates a cyclic sequence. Such a toy example tells us that our Algorithm \ref{alg1} possibly has superiority over some existing state-or-the-art primal-dual algorithms on some saddle point problems.
\end{remark}

\begin{remark}\label{remark-four}
	When assuming that both $f(x)$ and $g(y)$ are differentiable, the saddle point problem \eqref{Problem} can be also reformulated as a standard variational inequality problem (VIP), i.e., finding a point $\u^*\in\U$ such that
	\begin{equation*}\label{VIP}
		\langle \u-\u^*, \nabla\L(\u^*)\rangle \geq 0,\quad \forall \u\in\U,
	\end{equation*}
	where $\u$ and $\U$ are given in \eqref{MVI-b}, and $\nabla \L(\u)$ is specified as
	\begin{equation*}
		\nabla \L(\u)=\left(\begin{array}{c}
			\nabla_{x}\L(x,y) \\ -\nabla_{y}\L(x,y)
		\end{array}\right)=\left(\begin{array}{c}
			\nabla f(x)+A^\top y \\ \nabla g(y)-Ax
		\end{array}\right).
	\end{equation*}
	Consequently, we can gainfully employ some efficient algorithms, such as the projection methods and extragradient methods (e.g., \cite{Kor76,Nem04,Pop80,Zhang22}) tailored for VIPs to deal with those saddle point problems with differentiable objectives. However, many real-world problems do not necessarily satisfy the Lipschitz continuity of $\nabla \L(\u)$ required in convergence analysis. Besides, the aforementioned VIP-type methods such as the well-known extragradient methods (e.g., \cite{Kor76,Nem04,Pop80,Zhang22}) treat both $x$ and $y$ as a single entity, thereby potentially ignoring the distinct properties associated with each individual variable. As a consequence, when implementing the aforementioned extragradient methods, we must update $x$ and $y$ simultaneously, thereby reducing the implementability of extragradient methods for \eqref{Problem}. Comparatively, our Algorithm \ref{alg1} does not require the Lipschitz continuity of $\nabla \L(\u)$. More promisingly, our method is able to maximally exploit the structure of \eqref{Problem} so that each subproblem is easily implemented for many real-world problems (see Section \ref{Sec5}).
\end{remark}

\begin{remark}\label{remark-three}
	After we finished this manuscript, we were brought to the so-named Alternating Extragradient Method (AEM) introduced by Bonettini and Ruggiero \cite{BR11,BR14} and the Primal-Dual Fixed Point algorithm (PDFP) proposed by Chen et al. \cite{CHZ16}. Firstly, under the same differentiability  requirement on $\L(x,y)$, the AEM is an improved variant of the classical extragradient methods \cite{Kor76,Nem04,Pop80}, where the respective structure associated with $x$ and $y$ can be efficiently explored. Moreover, the AEM enjoys an effective adaptive stepsize for primal and dual subproblems. However, the Lipschitz continuity assumption for the AEM would possibly preclude its applicability to some nonsmooth real-world problems. Comparatively, although our Algorithm \ref{alg1} shares the similar symmetric spirit to update the primal and dual variables with the AEM, our Algorithm \ref{alg1} is easily applicable to nonsmooth saddle point problems with a bilinear coupling term. Moreover, as a theoretical complement, we not only prove the global convergence of Algorithm \ref{alg1}, but also estimate the linear convergence rate under the generalized error bound defined by Bregman distance. Secondly, when comparing with the PDFP in \cite{CHZ16}, we can see that PDFP shares a highly similar iterative scheme with our SPIDA by computing the $y$-subproblem twice at each iteration. In particular, when the proximal term in the $x$-subproblem \eqref{SPIDA-b} of SPIDA is chosen as the
		Euclidean norm (i.e., $\B_{\psi}(x,x^{k})=\frac{1}{2}\|x-x^k\|^2$), the iterative form for the $x$-subproblem coincides with the one of PDFP except the appearance of proximal parameter $\mu$. On the other hand, our SPIDA has versatile Bregman proximal terms making it flexible for many structured optimization problems. From this point, our SPIDA is more general than the PDFP.
\end{remark}

\subsection{Global convergence}
In this subsection, we will prove the global convergence of Algorithm \ref{alg1} under some standard conditions used in the primal-dual literature.

Note that both subproblems \eqref{SPIDA-a} and \eqref{SPIDA-c} share a similar form, so we denote
$$\by_w=\arg\max_{y\in\Y}\left\{-g(y)+\langle Aw,y\rangle -\gamma \B_{\phi}(y,y^k)\right\}$$
with some given $w\in\R^m$ for convenience, which can also be rewritten as
\begin{equation}\label{y-sub}
	\by_w=\arg\min_{y\in\Y}\left\{g(y)-\langle Aw,y\rangle +\gamma \phi(y) -\gamma \langle \nabla \phi(y^k),y-y^k\rangle\right\}.
\end{equation}
Hereafter, we begin our analysis with the following lemma.

\begin{lemma}\label{lemma1}
	Suppose that the Bregman kernel function $\phi(\cdot)$ is $\varrho$-strongly convex. Let $\by_u$ and $\by_w$ be solutions of \eqref{y-sub} for some given $u$ and $w$, respectively. Then, the following inequality
	\begin{align}\label{eq3}
		g(\by_u)-g(y) +\langle \by_u-y,-Aw \rangle \le
		&\;\gamma \B_{\phi}(y,y^k)- \gamma \B_{\phi}(y,\by_w)- \gamma \B_{\phi}(\by_u,y^k)  \nonumber\\
		&\; +  \frac{\alpha}{2} \| A^\top A \|  \|  w - u \|^{2} - \left(\gamma- \frac{1}{\alpha \varrho}\right)\B_{\phi}(\by_w,\by_u)  
	\end{align}
	holds for all $y\in\Y$, where $\alpha$ is a positive constant.
\end{lemma}

\begin{proof}
	Note that $\by_u$ is a minimizer of \eqref{y-sub} for given $u\in\R^m$. It then follows from the first-order optimality condition of \eqref{y-sub} that
	\begin{equation}\label{lem-ineq1}
		g(\by_u)-g(y)  \le \langle y - \by_u,-Au + \gamma\nabla \phi(\by_u) -\gamma\nabla \phi(y^k)\rangle,\;\; \forall y\in  \Y.
	\end{equation}
	Consequently, by using the arbitrariness of $y$ in \eqref{lem-ineq1} with setting $y=\by_w$, we have
	\begin{equation}\label{eq5}
		g(\by_u)-g(\by_w)  \le \langle \by_w - \by_u,-Au + \gamma\nabla \phi(\by_u) -\gamma\nabla \phi(y^k)\rangle.
	\end{equation}
	Similarly, by the definition of $\by_w$, it then follows from the first-order optimality condition of \eqref{y-sub} that
	\begin{equation}\label{eq4}
		g(\by_w)-g(y)  \le \langle y - \by_w,-Aw + \gamma\nabla \phi(\by_w) -\gamma\nabla\phi(y^k) \rangle, \;\;\forall y\in  \Y.
	\end{equation}
	Setting $y=\by_u$ in \eqref{eq4} arrives at
	\begin{equation}
		g(\by_w)-g(\by_u)  \le \langle \by_u - \by_w,-Aw +\gamma \nabla \phi(\by_w) -\gamma\nabla \phi(y^k)\rangle.
	\end{equation}
	By invoking the definition of Bregman distance, we have
	\begin{align}\label{eq16}
		&\gamma\B_{\phi}(y,y^k) - \gamma\B_{\phi}(y,\by_w)  \nonumber\\
		&=\gamma\phi(\by_w) - \gamma\phi(y^k) + \gamma\langle\nabla \phi(\by_w),y-\by_w \rangle - \gamma\langle \nabla \phi(y^k),y-y^k \rangle \nonumber \\
		&=\gamma\phi(\by_w) - \gamma\phi(y^k) + \langle Aw,y-\by_w \rangle + \gamma\langle \nabla \phi(y^k),y^k-\by_w \rangle  \nonumber \\
		&\hskip3.5cm+\langle \gamma\nabla \phi(\by_w) - \gamma\nabla \phi(y^k)- Aw,y-\by_w \rangle.  
	\end{align}
	Using inequality \eqref{eq4} instead of the last term of \eqref{eq16} leads to
	\begin{align}\label{dif_Breg}
		&\gamma\B_{\phi}(y,y^k) - \gamma\B_{\phi}(y,\by_w)  \nonumber \\
		&\geq \gamma\phi(\by_w) - \gamma\phi(y^k) + \langle Aw,y-\by_w \rangle + \gamma\langle \nabla \phi(y^k),y^k-\by_w \rangle + g(\by_w) - g(y) \nonumber\\
		& =\gamma \phi(\by_w) - \gamma \phi(y^k) + \langle Aw,\by_u-\by_w \rangle + \gamma\langle \nabla \phi(y^k),y^k-\by_w \rangle - \langle Aw,\by_u-y \rangle  + g(\by_w) - g(y) \nonumber \\
		&=\Phi(y^k,\by_u,\by_w) - \langle Aw,\by_u-y \rangle  + g(\by_w) - g(y),
	\end{align}
	where
	\begin{equation}\label{Phi}
		\Phi(y^k,\by_u,\by_w):=\gamma \phi(\by_w) - \gamma\phi(y^k) + \langle Aw,\by_u-\by_w \rangle + \gamma\langle \nabla \phi(y^k),y^k-\by_w \rangle.
	\end{equation}
	
	Below, we focus on $\Phi(y^k,\by_u,\by_w)$. First, we have
	\begin{align}\label{wyu}
		\langle Aw,\by_u-\by_w \rangle  & = \langle \by_u-\by_w, Aw - Au \rangle  + \langle \by_u-\by_w, Au \rangle \nonumber \\
		&=\langle \by_u-\by_w, Aw - Au \rangle +\langle \by_w- \by_u, \gamma \nabla \phi(y^k)-\gamma\nabla \phi(\by_u)\rangle \nonumber  \\
		&\hskip 3.6cm +\langle \by_w- \by_u, -Au + \gamma\nabla \phi(\by_u)-\gamma\nabla \phi(y^k)\rangle \nonumber \\
		&\geq \langle \by_u-\by_w, Aw - Au \rangle +\langle \by_w- \by_u, \gamma \nabla \phi(y^k)-\gamma\nabla \phi(\by_u)\rangle  + g(\by_u) -g(\by_w),
	\end{align}
	where the last inequality follows from \eqref{eq5}. As a result, substituting \eqref{wyu} into \eqref{Phi} immediately arrives at
	\begin{align}\label{eq8}
		\Phi(y^k,\by_u,\by_w) & =\gamma \phi(\by_w) - \gamma\phi(y^k)  + \gamma\langle \nabla \phi(y^k),y^k-\by_w \rangle+ \langle Aw,\by_u-\by_w \rangle \nonumber \\
		&\geq\gamma \phi(\by_w) - \gamma \phi(y^k)  + \gamma\langle \nabla \phi(y^k),y^k-\by_w \rangle + \langle \by_u-\by_w, Aw - Au \rangle \nonumber
		\\&\hskip1.7cm +\langle \by_w- \by_u,  \gamma\nabla \phi(y^k)-\gamma\nabla \phi(\by_u)\rangle + g(\by_u) -g(\by_w)  \nonumber\\
		&=\gamma \phi(\by_w) - \gamma \phi(y^k)   + \langle \by_u-\by_w, Aw - Au \rangle -\langle \gamma \nabla \phi(y^k),\by_u-y^k \rangle \nonumber \\
		&\hskip1.7cm -\langle \gamma \nabla \phi(\by_u),\by_w-\by_u \rangle + g(\by_u) -g(\by_w) \nonumber \\
		&=\gamma \B_{\phi}(\by_w,\by_u) + \gamma \B_{\phi}(\by_u,y^k) + \langle \by_u-\by_w,   Aw - Au \rangle +  g(\by_u) -g(\by_w),
	\end{align}
	where the last equality follows from the definition of Bregman distance. Since the Bregman kernel function $\phi(\cdot)$ is $\varrho$-strongly convex,  an application of the fact $\langle a,b\rangle \geq -\frac{1}{2\alpha}\|a\|^2 -\frac{\alpha}{2}\|b\|^2$ for all $a,b\in\mathbb{R}^n$ and $\alpha>0$ immediately yields
	\begin{align}\label{eq9}
		\langle \by_u-\by_w,   Aw - Au \rangle
		&\geq  -\frac{1}{2\alpha}\|  \by_w-\by_u \|^{2} - \frac{\alpha}{2} \| A^\top A \|    \|  w - u \|^{2}  \nonumber \\
		&\geq -\frac{1}{\alpha \varrho}\B_{\phi}(\by_w,\by_u) - \frac{\alpha}{2} \| A^\top A \|  \|  w - u \|^{2} .
	\end{align}
	
	Plugging \eqref{eq9} into \eqref{eq8}, we have
	\begin{equation*}
		\Phi(y^k,\by_u,\by_w)  \geq \left(\gamma - \frac{1}{\alpha\varrho}\right) \B_{\phi}(\by_w,\by_u) + \gamma \B_{\phi}(\by_u,y^k)-\frac{\alpha}{2} \| A^\top A \|  \|  w - u \|^{2}  +  g(\by_u) -g(\by_w),
	\end{equation*}
	which, together with \eqref{dif_Breg}, implies that
	\begin{align*}
		&\gamma\B_{\phi}(y,y^k) - \gamma\B_{\phi}(y,\by_w)  \nonumber \\
		&\geq\Phi(y^k,\by_u,\by_w) - \langle Aw,\by_u-y \rangle  + g(\by_w) - g(y) \\
		&\geq \left(\gamma - \frac{1}{\alpha\varrho}\right) \B_{\phi}(\by_w,\by_u) + \gamma \B_{\phi}(\by_u,y^k)-  \frac{\alpha}{2} \| A^\top A \|  \|  w - u \|^{2}  +  g(\by_u) - \langle Aw,\by_u-y \rangle   - g(y).
	\end{align*}
	Rearranging terms of the above inequality completes the proof.
\end{proof}

\begin{lemma}
	Let $\{(x^{k+1},\tilde y^{k+1},{y}^{k+1})\}$ be a sequence generated by Algorithm \ref{alg1}. Then, for all $x \in \mathcal{X}$ and $y \in \mathcal{Y}$, we have
	\begin{align}\label{eq11}
		\L(x^{k+1},y) - \L(x^{k+1},\tilde y^{k+1})
		&\leq  \gamma \B_{\phi}(y,y^{k})- \gamma \B_{\phi}(y,y^{k+1})  - \gamma\B_{\phi}(\tilde y^{k+1},y^{k})\nonumber \\
		&\quad +  \frac{\alpha}{2} \| A^\top A \|  \|  x^{k+1} - x^{k} \|^{2}  -\left(\gamma-\frac{1}{\alpha\varrho}\right)\B_{\phi}(y^{k+1},\tilde y^{k+1})
	\end{align}
	and
	\begin{equation}\label{eq12}
		\L(x^{k+1},\tilde y^{k+1}) - \L(x,\tilde y^{k+1}) \leq \mu\B_{\psi}(x,x^{k}) - \mu\B_{\psi}(x,x^{k+1}) - \mu\B_{\psi}(x^{k+1},x^{k}).
	\end{equation}
\end{lemma}

\begin{proof}
	First, it is clear from the notation in \eqref{Problem} that
	\begin{equation}\label{dif_Lxy}
		\L(x^{k+1},y) - \L(x^{k+1},\tilde y^{k+1}) =g(\tilde y^{k+1})-g(y) +\langle \tilde y^{k+1}-y,-Ax^{k+1} \rangle.
	\end{equation}
	Since $\tilde{y}^{k+1}$ and $y^{k+1}$ are solutions of \eqref{SPIDA-a} and \eqref{SPIDA-c}, respectively, it immediately follows from \eqref{dif_Lxy} and Lemma \ref{lemma1} with settings  $\by_u=\tilde y^{k+1},\by_w=y^{k+1},u=x^{k},w=x^{k+1}$ that 	
	\begin{align}\label{diff_Ly}
		\L(x^{k+1},y) - \L(x^{k+1},\tilde y^{k+1}) &=  g(\tilde y^{k+1})-g(y) +\langle \tilde y^{k+1}-y,-Ax^{k+1} \rangle  \nonumber \\
		& \leq  \gamma \B_{\phi}(y,y^{k})- \gamma \B_{\phi}(y,y^{k+1})  - \gamma\B_{\phi}(\tilde y^{k+1},y^{k})\nonumber \\
		&\quad + \frac{\alpha}{2} \| A^\top A \| \|  x^{k+1} - x^{k} \|^{2} -\left(\gamma-\frac{1}{\alpha\varrho}\right)\B_{\phi}(y^{k+1},\tilde y^{k+1}),
	\end{align}
	
	which is precisely the same as \eqref{eq11}. We proved the first assertion of this lemma.
	
	Below, we show \eqref{eq12}. It follows from the notation of $\L(x,y)$ in \eqref{Problem} that
	\begin{equation}\label{diff_Lx}
		\L(x,\tilde y^{k+1}) - \L(x^{k+1},\tilde y^{k+1})  = f(x)-f(x^{k+1})+\langle Ax-Ax^{k+1},\tilde{y}^{k+1}\rangle.
	\end{equation}
	On the other hand, the first-order optimality condition of \eqref{SPIDA-b} is
	\begin{equation}\label{x-optcond}
		f(x)-f(x^{k+1})+\left\langle x-x^{k+1},A^\top\tilde{y}^{k+1} +\mu\nabla \psi(x^{k+1})-\mu\nabla \psi(x^k)\right\rangle\geq 0,\;\; \forall x\in\X.
	\end{equation}
	Consequently, combining \eqref{diff_Lx} and \eqref{x-optcond} leads to
	\begin{align}\label{lem2-ineq2}
		\L(x^{k+1},\tilde y^{k+1}) -\L(x,\tilde y^{k+1})& = -f(x)+f(x^{k+1})-\langle Ax-Ax^{k+1},\tilde{y}^{k+1}\rangle \nonumber \\
		&\leq  \left\langle x - x^{k+1}, \mu\nabla \psi(x^{k+1})-\mu\nabla \psi(x^{k}) \right\rangle  \nonumber \\
		&=-\mu\left\langle x - x^{k+1}, \nabla \psi(x^{k})-\nabla \psi(x^{k+1}) \right\rangle
	\end{align}
	Applying the three-point property of Bregman distance \eqref{three-point} to \eqref{lem2-ineq2} with setting $a=x^{k+1}$, $b=x^k$, and $c=x$, we conclude that
	\begin{align*}
		\L(x^{k+1},\tilde y^{k+1}) - \L(x,\tilde y^{k+1})  &\le -\mu\langle x - x^{k+1}, \nabla \psi(x^{k}) -\nabla \psi(x^{k+1})\rangle \nonumber \\
		&=\mu\B_{\psi}(x,x^{k}) - \mu\B_{\psi}(x,x^{k+1}) - \mu\B_{\psi}(x^{k+1},x^{k}).
	\end{align*}
	This completes the proof of this lemma.
\end{proof}

Before presenting the global convergence theorem, we first make the following assumption.
\begin{assumption}\label{assum1}
	The Bregman kernel function $\phi(\cdot)$ associated with the $y$-subproblem is $\varrho$-strongly convex such that the proximal parameter $\gamma$ and constant $\alpha$ satisfying $\alpha \varrho\gamma \geq 1$. Moreover, the Bregman kernel function $\psi(\cdot)$ associated with the $x$-subproblem is $ \kappa $-strongly convex such that $ \kappa\mu > \alpha\|A^\top A\|$.
\end{assumption}

\begin{remark}\label{rm1}
	Note that the condition assumed in Assumption \ref{assum1} is closely related to the standard requirement in \cite{CP11}. By invoking the arbitrariness of $\alpha$ in the derivation of \eqref{eq9}, both conditions $\alpha\varrho\gamma \geq 1$ and $ \kappa\mu > \alpha\|A^\top A\|$ in Assumption \ref{assum1} imply that $\mu\gamma\kappa\varrho > \| A^\top A \|.$
	In this situation, when $\kappa=\varrho=1$, such a condition reduces to the standard requirement of the PDHG \cite{CP11}.
\end{remark}

With the above preparations, we now state our global convergence theorem.

\begin{theorem}
	Suppose that Assumption \ref{assum1} holds.	Let $\{(x^{k+1},\tilde y^{k+1},{y}^{k+1})\}$ be a sequence generated by Algorithm \ref{alg1}. Then, for all $x \in \X$, $y \in \Y$, we have
	\begin{equation*}
		\L(\widehat{x}^{N},y) - \L(x,\widehat{y}^{N})  \le \frac{1}{N} \left(\gamma\B_{\phi}(y,y^{0})+\mu\B_{\psi}(x,x^{0})\right),
	\end{equation*}
	where $N$ is a positive integer, $\widehat{x}^{N}=\frac{1}{N} \sum_{k=0}^{N-1} x^{k}$ and $\widehat{y}^{N}=\frac{1}{N} \sum_{k=0}^{N-1} \tilde y^{k}$.
\end{theorem}

\begin{proof}
	Under Assumption \ref{assum1}, it is easy to check by Lemma \ref{lem:Bregman} that there exists a constant $\widehat{\beta} := (\kappa\mu - \alpha\|A^\top A\|)/(2\mu) >0$ such that
	\begin{align}\label{ineqB}
		\B_{\psi}(x^{k+1},x^{k}) - \frac{\alpha}{2\mu} \|A^\top A \| \|  x^{k+1} - x^{k} \|^{2} \ge \widehat{\beta}\|  x^{k+1} - x^{k} \|^{2} >  0.
	\end{align}
	Consequently, by adding \eqref{eq11} and \eqref{eq12}, it follows from \eqref{ineqB} and the positivity of the Bregman distance ($\alpha\varrho\gamma \ge 1$) that
	\begin{align}\label{eq13}
		&\L(x^{k+1}, y) - \L(x,\tilde y^{k+1})  \nn\\
		&\le\gamma \B_{\phi}(y,y^{k})- \gamma \B_{\phi}(y,y^{k+1})+ \mu\B_{\psi}(x,x^{k}) - \mu\B_{\psi}(x,x^{k+1}) \nonumber \\
		&\;\;-\mu\left(\B_{\psi}(x^{k+1},x^{k}) - \frac{\alpha}{2\mu} \|A^\top A \| \|  x^{k+1} - x^{k} \|^{2}\right)  -\gamma \B_{\phi}(\tilde y^{k+1},y^{k}) - \left(\gamma-\frac{1}{\alpha\varrho}\right)\B_{\phi}(y^{k+1},\tilde y^{k+1})  \nonumber \\
	 &	\le \gamma \B_{\phi}(y,y^{k})- \gamma \B_{\phi}(y,y^{k+1})+ \mu\B_{\psi}(x,x^{k}) - \mu\B_{\psi}(x,x^{k+1}) \nonumber \\
		&\;\; - \widehat{\beta} \|  x^{k+1} - x^{k} \|^{2}
		-\gamma \B_{\phi}(\tilde y^{k+1},y^{k})  - \left(\gamma-\frac{1}{\alpha \varrho}\right)\B_{\phi}(y^{k+1},\tilde y^{k+1})   \nonumber \\
		& \le\gamma \B_{\phi}(y,y^{k}) - \gamma\B_{\phi}(y,y^{k+1})+ \mu\B_{\psi}(x,x^{k}) - \mu\B_{\psi}(x,x^{k+1}).
	\end{align}
	Hence, summing up \eqref{eq13} from $k=0$ to $N-1$ leads to
	\begin{align}\label{sum-ineq}
		&\sum_{k=0}^{N-1}\left(\L(x^{k+1}, y) - \L(x,\tilde y^{k+1})\right) \\
		& \le \gamma \B_{\phi}(y,y^{0})-
		\gamma\B_{\phi}(y,y^{N})+ \mu\B_{\psi}(x,x^{0}) - \mu\B_{\psi}(x,x^{N}) \nonumber\\
		&\le  \gamma \B_{\phi}(y,y^{0})+ \mu\B_{\psi}(x,x^{0}) . \nonumber
	\end{align}
	Notice that $\L(x,y)$ and $-\L(x,y)$ are convex with respect to $x$ and $y$, respectively. It then follows from the Jensen inequality and \eqref{sum-ineq} that
	\begin{align*}
		\L(\widehat{x}^{N}, y) - \L(x, \widehat{y}^{N})
		& \le\frac{1}{N} \sum_{k=0}^{N-1}\left(\L(x^{k+1}, y) -\L(x,\tilde y^{k+1}) \right)\nonumber \\
		& \le \frac{1}{N}\left( \gamma \B_{\phi}(y,y^{0})+ \mu \B_{\psi}(x,x^{0})\right). \nonumber
	\end{align*}
	We complete the proof of this theorem.
\end{proof}

\begin{theorem}
	Suppose that Assumption \ref{assum1} holds.	The sequence $\{(x^{k+1},{y}^{k+1})\}$ generated by Algorithm \ref{alg1} converges to a saddle point of $\L(x,y)$ in \eqref{Problem}.
\end{theorem}

\begin{proof}
	Letting $(x^{*},y^{*})$ be a saddle point of \eqref{Problem}, we immediately have
	\begin{equation}\label{saddleineq}
		\L(x^{*}, y) \le \L(x^{*}, y^{*}) \le \L(x, y^{*}), \quad \forall x\in\X,\;\forall y\in\Y.
	\end{equation}
	On the other hand, by setting $x=x^{k+1}$ and $y=\tilde{y}^{k+1}$ in \eqref{saddleineq}, it follows from \eqref{eq13} with setting $x=x^{*}$ and $y=y^{*}$ that
	\begin{align*}
		0 &\le \L(x^{k+1}, y^{*}) - \L(x^{*},\tilde y^{k+1})\\
		&\le  \gamma \B_{\phi}(y^{*},y^{k})- \gamma\B_{\phi}(y^{*},y^{k+1})  + \mu\B_{\psi}(x^{*},x^{k}) - \mu\B_{\psi}(x^{*},x^{k+1}),
	\end{align*}
	which clearly implies that
	\begin{equation}\label{noninc}
		0 \le \gamma\B_{\phi}(y^{*},y^{k+1})+\mu\B_{\psi}(x^{*},x^{k+1}) \le \gamma\B_{\phi}(y^{*},y^{k})+ \mu\B_{\psi}(x^{*},x^{k}).
	\end{equation}
	Obviously, inequality \eqref{noninc} means that the sequence $\{\gamma\B_{\phi}(y^{*},y^{k})+ \mu\B_{\psi}(x^{*},x^{k})\}$ is Bregman monotone (see \cite{BBC03}) decreasing and bounded. By the strong convexity of the Bregman distance, the sequence $\{(x^{k},y^{k}) \}$ is bounded. Therefore, there exists a subsequence $\{(x^{k_{j}},y^{k_{j}})\}$ converging to a cluster point, denoted by $(x^{\infty},y^{\infty})$. Below, we aim to show that such a cluster point is a saddle point of \eqref{Problem}.
	
	According to \eqref{eq13}, we also have
	\begin{align}\label{fejer331}
		&\gamma\B_{\phi}(\tilde y^{k+1},y^{k}) +  \left(\gamma-\frac{1}{\alpha \varrho}\right)\B_{\phi}(y^{k+1},\tilde y^{k+1}) +  \widehat{\beta} \|  x^{k+1} - x^{k} \|^{2} \nonumber	\\
		&\quad \le  \gamma\B_{\phi}(y^{*},y^{k})- \gamma\B_{\phi}(y^{*},y^{k+1}) + \mu\B_{\psi}(x^{*},x^{k})  - \mu\B_{\psi}(x^{*},x^{k+1}).
	\end{align}
	Summing the above inequality from $k=0$ to $N$, since $\alpha >0$, $\gamma>0$, $\beta>0$ and $\alpha \gamma \varrho \ge 1$ are constant, we have
	\begin{align*}
		&\sum\limits_{k=0}^{N}\left(\gamma\B_{\phi}(\tilde y^{k+1},y^{k}) + \left(\gamma-\frac{1}{\alpha \varrho}\right)\B_{\phi}(y^{k+1},\tilde y^{k+1}) +  \widehat{\beta} \|  x^{k+1} - x^{k} \|^{2} \right) \\
		&\le  \gamma\B_{\phi}(y^{*},y^{0})+ \mu\B_{\psi}(x^{*},x^{0}),
	\end{align*}
	which implies that
	\begin{equation}\label{eq14}
		\lim_{k \rightarrow \infty}\B_{\phi}(\tilde y^{k+1},y^{k}) =\lim_{k \rightarrow \infty}\B_{\phi}(y^{k+1},\tilde y^{k+1}) =
		\lim_{k \rightarrow \infty}  \|  x^{k+1} - x^{k} \|^{2} =0.
	\end{equation}
	Moreover, we conclude from \eqref{eq14} that both $\{y^k\}$ and $\{\tilde{y}^k\}$ converge to the same point. Therefore, $(x^{\infty},y^{\infty})$ is also a cluster point of the subsequence $\{(x^{k_{j}},\tilde y^{k_{j}})\}$.
	It then follows from the first-order optimality conditions of \eqref{SPIDA-b} and \eqref{SPIDA-c} that, for all $x\in\X$ and $y\in\Y$, the following inequalities hold
	\begin{equation*}
		\left\{\begin{aligned}
			&f(x)-f(x^{k+1})+\left\langle x-x^{k+1},\;A^\top\tilde{y}^{k+1} +\mu\nabla \psi(x^{k+1})-\mu\nabla \psi(x^k)\right\rangle\geq 0, \\
			&g(y)-g(y^{k+1})+\left\langle y-y^{k+1},\;-Ax^{k+1}+\gamma \nabla\phi(y^{k+1}) -\gamma \nabla\phi(y^k)\right\rangle \geq 0.
		\end{aligned}\right.
	\end{equation*}
	Therefore, for all $x\in\X$ and $y\in\Y$, the subsequence $\{(x^{k_{j}},\tilde{y}^{k_j},y^{k_{j}})\}$ satisfies
	\begin{equation}\label{subseq}
		\left\{\begin{aligned}
			&f(x)-f(x^{k_j+1})+\left\langle x-x^{k_j+1},\;A^\top\tilde{y}^{k_j+1} +\mu\nabla \psi(x^{k_j+1})-\mu\nabla \psi(x^{k_j})\right\rangle\geq 0, \\
			&g(y)-g(y^{k_j+1})+\left\langle y-y^{k_j+1},\;-Ax^{k_j+1}+\gamma \nabla\phi(y^{k_j+1}) -\gamma \nabla\phi(y^{k_j})\right\rangle \geq 0.
		\end{aligned}\right.
	\end{equation}
	Consequently, taking limit as $j\to\infty$ over the subsequence ${k_j}$ in \eqref{subseq}, it follows from \eqref{eq14} that
	\begin{equation*}
		\left\{\begin{aligned}
			&f(x)-f(x^{\infty})+\left\langle x-x^\infty,\; A^\top y^\infty \right\rangle\geq 0, \quad \forall x\in\X,\\
			&g(y)-g(y^\infty)+\left\langle y-y^\infty,\; -Ax^\infty\right\rangle \geq 0,\quad \forall y\in\Y,
		\end{aligned}\right.
	\end{equation*}
	which, together with \eqref{KKT}, means that the limit point $(x^{\infty},y^{\infty})$ is a saddle point of $\L(x,y)$ in \eqref{Problem}. So $(x^{*}, y^{*})$ can be replaced by $(x^{\infty},y^{\infty})$ in the the sequence $\{\gamma\B_{\phi}(y^{*},y^{k})+ \mu\B_{\psi}(x^{*},x^{k})\}$. Thus $\{\gamma\B_{\phi}(y^{\infty},y^{k})+ \mu\B_{\psi}(x^{\infty},x^{k})\}$ is monotone decreasing and bounded, we obtain
	\begin{equation*}
		\lim_{k \rightarrow \infty} \left(\gamma\B_{\phi}(y^{\infty},y^{k})+ \mu\B_{\psi}(x^{\infty},x^{k})\right) =\lim_{j \rightarrow \infty} \left(\gamma\B_{\phi}(y^{\infty},y^{k_{j}})+ \mu\B_{\psi}(x^{\infty},x^{k_{j}})\right) = 0,
	\end{equation*}
	which means that $x^{k} \rightarrow x^{\infty}$ and $y^{k} \rightarrow y^{\infty}$ as $k \rightarrow \infty$. We complete the proof of this theorem.
\end{proof}

\subsection{Linear convergence}\label{sec_linear}
In this subsection, we turn our attention to establishing the linear convergence in the context of a generalized error bound conditions. First, we make the following assumption, which is an extended version used in the literature, e.g., \cite{JWCZ21,JZH23,WH20,YH16}.

\begin{assumption}\label{ass36}
	Assume that, for any $\omega>0$, there exists $\eta>0$ such that
	\begin{equation}\label{err}
		\dist_{\B}(\u,\U^\star)\leq \eta \dist(0,\E(\u,1)), \;\;\;\forall \|\u\|\leq\omega,\; \u\in \U,
	\end{equation}
	where $\E(\u,1)$ is given by \eqref{error25} with $t=1$. Let $\gamma >0$, $\mu >0$ and $\dist_{\B}(\u,\U^\star)$  be defined by
	\begin{equation}\label{defB}
		\dist_{\B}(\u,\U^\star) =\min \left\{\sqrt{\B(\u^\star,\u)}\;\big{|}\; \u^\star \in \U^\star \right\}
	\end{equation}
	with $\B(\u^\star,\u):=\gamma\B_{\phi}(y^\star,y)+\mu \B_{\psi}(x^\star,x)$. Moreover, we assume that the gradients of $\psi(\cdot)$ and $\phi(\cdot)$ are Lipschitz continuous with modulus $L_\psi$ and $L_\phi$, respectively.
\end{assumption}

Hereafter, we establish the global linear convergence of Algorithm \ref{alg1} under Assumption \ref{ass36} in the context of a generalized error bound condition defined by Bregman distance. We begin our analysis with the following lemma.

\begin{lemma}
	Let $\{(x^{k+1},\tilde y^{k+1}, y^{k+1})\}$ be a sequence generated by Algorithm \ref{alg1}. Then, we have
	\begin{align}\label{ineq335}
		\dist^2\left(0,\E(\u^{k+1},1)\right) &\leq 2\left(\|AA^\top\|+\gamma^2 L_{\phi}^2\right)\|\tilde{y}^{k+1}-y^{k+1}\|^2  \nn \\
		& \qquad +2\mu^2 L_\psi^2\|x^{k+1}-x^{k}\|^2 + 2\gamma^2 L_{\phi}^2 \left\|\tilde{y}^{k+1}-y^{k}\right\|^2.
	\end{align}
	where $L_{\phi}$ and $L_\psi$ are the Lipschitz continuity constants of $\nabla\phi(\cdot)$ and $\nabla\psi(\cdot)$, respectively.
\end{lemma}

\begin{proof}
	Letting $\xi^{k+1}\in\partial f(x^{k+1})$, it then follows from the first-order optimality condition of $x$-subproblem \eqref{SPIDA-b} that
	\begin{equation}\label{eq:lem3.3a}
		x^{k+1}=\Pi_{\X} \left[x^{k+1}-\left(\xi^{k+1}-A^\t \tilde{y}^{k+1} +\mu\nabla \B_{\psi}(x^{k+1},x^{k})\right)\right],
	\end{equation}
	where $\nabla \B_{\psi}(x^{k+1},x^{k}) =  \langle \nabla \psi(x^{k+1}) - \nabla \psi(x^{k}),x^{k+1}-x^{k} \rangle $. By using \eqref{eq:lem3.3a} and the nonexpansiveness of projection operator $\Pi_{\X}(\cdot)$, we have
	\begin{align}\label{eq424}
		&\dist^2\left(0, \E_{\X}(\u^{k+1},1)\right) \nn \\
		& =  \dist^2\left(x^{k+1}, \Pi_{\X} \left[x^{k+1}-(\partial f(x^{k+1})+A^\t y^{k+1})\right]\right) \nn \\
		& \leq  \left\|\Pi_{\X} \left[x^{k+1}-\left(\xi^{k+1}+A^\t \tilde{y}^{k+1} +\mu\nabla\B_{\psi}(x^{k+1},x^{k})\right)\right]  -\Pi_{\X} \left[x^{k+1}-(\xi^{k+1}+A^\t y^{k+1})\right] \right\|^2 \nn \\
		&\leq  \left\|A^\t(y^{k+1}-\tilde{y}^{k+1})-\mu\nabla\B_{\psi}(x^{k+1},x^{k})\right\|^2 \nn \\
		&\leq  2\|AA^\t\|  \left\|\tilde{y}^{k+1}-y^{k+1}\right\|^2+2\mu^2 \left\|\nabla\B_{\psi}(x^{k+1},x^{k})\right\|^2 \nn \\
		&\leq 2\|AA^\t\|  \left\|\tilde{y}^{k+1}-y^{k+1}\right\|^2+2\mu^2 L_\psi^2\left\|x^{k+1}-x^{k}\right\|^2,
	\end{align}
	where the second inequality is derived by the fact that $ \|a+b\|^2\leq 2\|a\|^2+2\|b\|^2$ holds for all $ a, b\in \R^n$,
	and the last inequality follows from the Lipschitz continuity of $\nabla \psi(\cdot)$.
	
	Similarly, it follows from the first-order optimality condition of the $y$-subproblem \eqref{SPIDA-c} that
	$$ y^{k+1}= \Pi_{\Y}\left[y^{k+1}-\left(\zeta^{k+1}-Ax^{k+1}+\gamma\nabla\B_{\phi}(y^{k+1},y^{k})\right)\right].$$
	Then, we have
	\begin{align}\label{eq425}
		&\dist^2\left(0, \E_{\Y}(\u^{k+1},1)\right) \nn \\
		&= \dist^2\left( y^{k+1}, \Pi_{\Y}\left[y^{k+1}-\left(\partial g(y^{k+1})-Ax^{k+1}\right)\right] \right) \nn \\
		&\leq \left\|\Pi_{\Y} \left[ y^{k+1}-\left(\zeta^{k+1}-Ax^{k+1}+\gamma \nabla\B_{\phi}(y^{k+1},y^{k}) \right) \right]  -\Pi_{\Y} \left[ y^{k+1}-(\zeta^{k+1}-Ax^{k+1})\right] \right\|^2 \nn \\
		&\leq \gamma^2 \left\|\nabla\B_{\phi}(y^{k+1},y^{k})\right\|^2 \nn \\
		& \leq  \gamma^2 L_{\phi}^2 \left\|y^{k+1}-y^{k}\right\|^2 \nn \\
		&    \leq  2\gamma^2 L_{\phi}^2 \left\|y^{k+1}-\tilde{y}^{k+1}\right\|^2  + 2\gamma^2 L_{\phi}^2 \left\|\tilde{y}^{k+1}-y^{k}\right\|^2.
	\end{align}
	Consequently, combining \eqref{eq424} and \eqref{eq425} leads to
	\begin{align*}
		&\dist^2\left(0,\E(\u^{k+1},1)\right)\nn  \\
		&=\dist^2\left(0, \E_{\X}(\u^{k+1},1)\right)+ \dist^2(0, \E_{\Y}\left(\u^{k+1},1)\right) \nn \\
		&\leq   2 \left(\|AA^\t\|+\gamma^2 L_{\phi}^2\right)\|\tilde{y}^{k+1}-y^{k+1}\|^2+2\mu^2 L_\psi^2\|x^{k+1}-x^{k}\|^2    + 2\gamma^2 L_{\phi}^2 \left\|\tilde{y}^{k+1}-y^{k}\right\|^2.
	\end{align*}
	Hence, the assertion of this lemma is obtained.
\end{proof}

Now, we establish the linear convergence rate of Algorithm \ref{alg1} by the following theorem.
\begin{theorem}
	Let $\{(x^{k+1},\tilde y^{k+1},{y}^{k+1})\}$ be the sequence generated by Algorithm \ref{alg1}. Suppose that Assumptions \ref{assum1} and \ref{ass36} hold. Then, we have
	\begin{equation}
		(1+\vartheta)\dist_{\B}^2(\u^{k+1},\U^\star)\le \dist_{\B}^2(\u^k,\U^\star), \notag
	\end{equation}
	where $\vartheta$ is a positive constant given by
	$$ \vartheta=\min \left\{\;\frac{\varrho}{4\gamma\eta^2  L_{\phi}^2},\; \frac{\gamma\alpha \varrho-1 }{4\alpha \eta^2 (\|AA^\top\|+ \gamma^2L_{\phi}^2)}, \;   \frac{  \kappa    }{4\mu \eta^2 L_\psi^2} \; \right\}.$$
\end{theorem}

\begin{proof} It first follows from \eqref{fejer331} that
	\begin{align}\label{ieq339}
		\gamma\B_{\phi}(y^\star,y^{k+1})+\mu \B_{\psi}(x^\star,x^{k+1})
		& \le \gamma \B_{\phi}(y^\star,y^{k})+\mu \B_{\psi}(x^\star,x^{k}) - \gamma\B_{\phi}(\tilde y^{k+1},y^{k})\nn \\
		& \quad  - \left(\gamma-\frac{1}{\alpha \varrho}\right)\B_{\phi}(y^{k+1},\tilde y^{k+1}) - \widehat{\beta} \|  x^{k+1} - x^{k} \|^{2} .
	\end{align}
	In accordance with Assumption \ref{ass36}, i.e., inequality \eqref{err}, it follows from \eqref{ineq335} that
	\begin{align}
		&\dist_{\B}^2\left(\u^{k+1},\U^\star\right) \nn \\
		& \leq  \eta^2   \dist^2\left(0,\E(\u^{k+1},1)\right) \nn \\
		& \leq \eta^2 \left[ 2\left(\|AA^\t\|+\gamma^2 L_{\phi}^2\right) \left\|\tilde{y}^{k+1}-y^{k+1}\right\|^2+2\mu^2 L_\psi^2 \left\|x^{k+1}-x^{k}\right\|^2   + 2\gamma^2 L_{\phi}^2 \left\|\tilde{y}^{k+1}-y^{k}\right\|^2 \right]. 
	\end{align}
	Consequently, by the definition of $\dist_{\B}(\u,\U^\star)$ given by \eqref{defB}, we have
	\begin{align}\label{lc-ineq}
		& (1+\vartheta)\dist_{\B}^2(\u^{k+1},\U^\star)  \nn\\
		& = \dist_{\B}^2(\u^{k+1},\U^\star) + \vartheta \dist_{\B}^2(\u^{k+1},\U^\star) \nn \\
		&\leq \left[\gamma\B_{\phi}(y^\star,y^{k+1})+\mu \B_{\psi}(x^\star,x^{k+1})\right]+\vartheta\dist_{\B}^2(\u^{k+1},\U^\star) \nn  \\
		& \leq \gamma\B_{\phi}(y^\star,y^{k})+\mu \B_{\psi}(x^\star,x^{k})   - \gamma\B_{\phi}(\tilde y^{k+1},y^{k}) - \left(\gamma-\frac{1}{\alpha \varrho}\right)\B_{\phi}(y^{k+1},\tilde y^{k+1}) - \widehat{\beta} \|  x^{k+1} - x^{k} \|^{2} \nn\\
		&\quad + \vartheta \eta^2  \left[ 2\left(\|AA^\t\|+\gamma^2 L_{\phi}^2\right) \left\|\tilde{y}^{k+1}-y^{k+1}\right\|^2+2\mu^2 L_\psi^2\left\|x^{k+1}-x^{k}\right\|^2 + 2\gamma^2 L_{\phi}^2 \left\|\tilde{y}^{k+1}-y^{k}\right\|^2 \right] \nn \\
		& \leq \gamma\B_{\phi}(y^\star,y^{k})+\mu \B_{\psi}(x^\star,x^{k})-\left[\frac{\gamma \varrho}{2}-2 \vartheta \eta^2 \gamma^2 L_{\phi}^2 \right] \left\|\tilde{y}^{k+1}-y^{k}\right\|^2  \nn\\
		& \quad-\left[\frac{\gamma\alpha\varrho-1}{2\alpha}  -2\vartheta \eta^2   \left(\|AA^\t\|+\gamma^2 L_{\phi}^2\right)\right] \left\|\tilde{y}^{k+1}-y^{k+1}\right\|^2 \nn \\
		& \quad - \left[ \frac{  \widehat{\beta} \kappa }{2}-2\vartheta \eta^2\mu^2 L_\psi^2 \right] \left\|x^{k+1}-x^{k}\right\|^2, 
	\end{align}
	where the second inequality is derived by \eqref{ieq339}, and the third inequality follows from the strong convexity of $\B_{\phi}$
	and $\B_{\psi}$. As a consequence, by using the definition of $\vartheta$, inequality \eqref{lc-ineq} implies that
	\begin{equation}\label{finineq}
		(1+\vartheta)\dist_{\B}^2(\u^{k+1},\U^\star) \leq \gamma\B_{\phi}(y^\star,y^{k})+\mu \B_{\psi}(x^\star,x^{k}) - \Gamma(x^{k+1},\tilde{y}^{k+1},y^{k+1}),
	\end{equation}
	where $\Gamma(x^{k+1},\tilde{y}^{k+1},y^{k+1})$ is a positive term composed by the last three terms of \eqref{lc-ineq}.
	Then, setting $ \dist_{\B}^2(\u^{k},\U^\star)=\gamma\B_{\phi}(y^\star,y^{k})+\mu \B_{\psi}(x^\star,x^{k})$ in \eqref{finineq}
	immediately arrives at
	\begin{equation*}
		(1+\vartheta)\dist_{\B}^2(\u^{k+1},\U^\star) \leq \gamma\B_{\phi}(y^\star,y^{k})+\mu \B_{\psi}(x^\star,x^{k}) = \dist_{\B}^2(\u^{k},\U^\star).
	\end{equation*}
	The assertion of this theorem is obtained.
\end{proof}

\section{Applications to Linearly Constrained Convex Minimization}\label{Sec-LCM}
In this section, we shall show that our Algorithm \ref{alg1} will recover the iterative schemes of some classical first-order algorithms, when applying it to linearly constrained convex optimization problems.

\subsection{One-block case: Augmented Lagrangian method}\label{sec-one-block}
In this subsection, we consider the following one-block convex minimization problem with linear constraints:
\begin{equation}\label{one-block-min}
	\min_{x} \left\{\; f(x) \;|\; Ax=b,\;\;x\in\X\subseteq \R^n\right\},
\end{equation}
where $f(\cdot):\X\to (-\infty,+\infty]$ is a proper closed convex function and $A\in\R^{m\times n}$ is a given matrix, and $b\in\R^m$ is a given vector.
Accordingly,  its augmented Lagrangian function reads as
\begin{equation*}\label{ALF}
	\L_{\gamma}(x,y)=f(x) + \langle Ax -b, y\rangle +\frac{1}{2\gamma}\|Ax-b\|^2, \quad \gamma >0,
\end{equation*}
and the augmented Lagrangian method for \eqref{one-block-min} is
\begin{equation*}
	\left\{ \begin{aligned}
		x^{k+1}&=\arg\min_{x\in\X}\left\{ f(x) + \langle Ax-b,y^k\rangle +\frac{1}{2\gamma}\|Ax-b\|^2 \right\}, \\
		y^{k+1} &= y^k +\frac{1}{\gamma}(Ax^{k+1}-b).
	\end{aligned}\right.
\end{equation*}
or equivalently,
\begin{equation}\label{ALM}
	\left\{ \begin{aligned}
		x^{k+1}&=\arg\min_{x\in\X}\left\{ f(x) + \frac{1}{2\gamma}\|Ax-b+\gamma y^k\|^2 \right\}, \\
		y^{k+1} &= y^k +\frac{1}{\gamma}(Ax^{k+1}-b).
	\end{aligned}\right.
\end{equation}
It is not difficult to observe that directly solving the $x$-subproblem in \eqref{ALM} is not an easy task (at least is not easily implementable), when $f(x)$ is a nonsmooth function (e.g., $\|x\|_1$ or nuclear norm for matrices) and $A$ is a general matrix, even for the case $\X=\R^n$ (e.g., see \cite{YY13a}). Therefore, a natural way to make \eqref{ALM} implementable is linearizing the quadratic penalty term at $x^k$ as follows:
\begin{equation}\label{linearize}
	\frac{1}{2\gamma}\|Ax-b\|^2 \approx \frac{1}{2\gamma}\|Ax^k-b\|^2 + \frac{1}{\gamma} \langle A^\top (Ax^k-b), x-x^k\rangle +\frac{\mu}{2}\|x-x^k\|^2,
\end{equation}
where $\mu\gamma  \geq\|A^\top A\|$.
Consequently, using the above approximation \eqref{linearize} instead of quadratic penalty term in \eqref{ALM}, we immediately obtain the so-called Linearized Augmented Lagrangian Method (LALM) for \eqref{one-block-min}, i.e.,
\begin{equation}\label{LALM}
	\left\{ \begin{aligned}
		x^{k+1}&=\arg\min_{x\in\X}\left\{ f(x) +\frac{\mu}{2}\left\|x -\left( x^k + \frac{1}{\mu}A^\top \left(y^k + \frac{1}{\gamma}(Ax^k-b)\right)\right)\right\|^2 \right\}, \\
		y^{k+1} &= y^k +\frac{1}{\gamma}(Ax^{k+1}-b).
	\end{aligned}\right.
\end{equation}

When applying our Algorithm \ref{alg1} to \eqref{one-block-min}, we first reformulate \eqref{one-block-min} as the following saddle point problem:
\begin{equation}\label{one-saddle}
	\min_{x\in\X} \max_{y\in \R^m} \left\{\L(x,y):=f(x)+\langle Ax, y\rangle - \langle b, y\rangle \right\}.
\end{equation}
Then, the specific iterative scheme of Algorithm \ref{alg1} for \eqref{one-saddle} reads as
\begin{equation}\label{PD-ALM}
	\left\{ \begin{aligned}
		\tilde{y}^{k+1} &= \arg\max_{y\in \R^m}\left\{ -\langle b,y\rangle + \langle Ax^k, y\rangle - \gamma \B_\phi(y,y^k) \right\},  \\
		x^{k+1} &= \arg\min_{x\in \mathcal{X}}\left\{ f(x) + \langle Ax, \tilde{y}^{k+1} \rangle +\mu\B_\psi(x,x^k) \right\},   \\
		y^{k+1} &= \arg\max_{y\in \R^m}\left\{ -\langle b,y\rangle + \langle Ax^{k+1}, y\rangle - \gamma \B_\phi(y,y^k) \right\}.
	\end{aligned}\right.
\end{equation}
Clearly, by setting the Bregman kernel functions as $\phi(y)=\frac{1}{2}\|y\|^2$ and $\psi(x)=\frac{1}{2}\|x\|^2_{A^\top A}$, the iterative scheme \eqref{PD-ALM} reads as
\begin{equation*}
	\left\{ \begin{aligned}
		\tilde{y}^{k+1} &= \arg\max_{y\in \R^m}\left\{ -\langle b,y\rangle + \langle Ax^k, y\rangle - \frac{\gamma}{2}\|y-y^k\|^2 \right\},  \\
		x^{k+1} &= \arg\min_{x\in \mathcal{X}}\left\{ f(x) + \langle Ax, \tilde{y}^{k+1} \rangle +\frac{\mu}{2}\|Ax-Ax^k\|^2 \right\},   \\
		y^{k+1} &= \arg\max_{y\in \R^m}\left\{ -\langle b,y\rangle + \langle Ax^{k+1}, y\rangle -  \frac{\gamma}{2}\|y-y^k\|^2 \right\},
	\end{aligned}\right.
\end{equation*}
which, by using the first-order optimality conditions of both $y$-subproblems,  can be immediately simplified as
\begin{equation*}\label{SPD-ALM}
	\left\{ \begin{aligned}
		\tilde{y}^{k+1} &= y^k+ \frac{1}{\gamma}(Ax^k-b),  \\
		x^{k+1} &= \arg\min_{x\in \mathcal{X}}\left\{ f(x) + \frac{\mu}{2}\left\|Ax-\left( Ax^k -\frac{1}{\mu} \tilde{y}^{k+1}\right)\right\|^2 \right\},   \\
		y^{k+1} &= y^k+ \frac{1}{\gamma}(Ax^{k+1}-b).
	\end{aligned}\right.
\end{equation*}
It is trivial that substituting $\tilde{y}^{k+1}$ into the update of $x^{k+1}$ and setting $\mu=1/\gamma$ immediately yields the ALM \eqref{ALM}.

On the other hand, by setting the Bregman kernel functions as $\phi(y)=\frac{1}{2}\|y\|^2$ and $\psi(x)=\frac{1}{2}\|x\|^2$, the iterative scheme \eqref{PD-ALM} reads as
\begin{equation*}
	\left\{ \begin{aligned}
		\tilde{y}^{k+1} &= \arg\max_{y\in \R^m}\left\{ -\langle b,y\rangle + \langle Ax^k, y\rangle - \frac{\gamma}{2}\|y-y^k\|^2 \right\},  \\
		x^{k+1} &= \arg\min_{x\in \mathcal{X}}\left\{ f(x) + \langle Ax, \tilde{y}^{k+1} \rangle +\frac{\mu}{2}\|x-x^k\|^2 \right\},   \\
		y^{k+1} &= \arg\max_{y\in \R^m}\left\{ -\langle b,y\rangle + \langle Ax^{k+1}, y\rangle -  \frac{\gamma}{2}\|y-y^k\|^2 \right\},
	\end{aligned}\right.
\end{equation*}
which, by using the first-order optimality conditions of both $y$-subproblems,  can be immediately simplified as
\begin{equation*}\label{PD-LALM}
	\left\{ \begin{aligned}
		\tilde{y}^{k+1} &= y^k+ \frac{1}{\gamma}(Ax^k-b),  \\
		x^{k+1} &= \arg\min_{x\in \mathcal{X}}\left\{ f(x) + \frac{\mu}{2}\left\|x-\left( x^k -\frac{1}{\mu}A^\top \tilde{y}^{k+1}\right)\right\|^2 \right\},   \\
		y^{k+1} &= y^k+ \frac{1}{\gamma}(Ax^{k+1}-b).
	\end{aligned}\right.
\end{equation*}
Clearly, plugging the formula of $\tilde{y}^{k+1}$ into the update scheme of $x^{k+1}$ immediately yields the LALM \eqref{LALM}.

It is interesting to notice that our algorithmic framework allows us to take different Bregman kernel functions. Therefore, we here follow the novel idea of the newly introduced balanced augmented Lagrangian method \cite{HY21} to further consider taking $\phi(y)=\frac{1}{2}\|y\|^2_{(AA^\top  +\epsilon I)}$ and $\psi(x)=\frac{1}{2}\|x\|^2$ for model \eqref{one-block-min}, where $I$ stands for an identity matrix which always automatically matches the size for operations, and $I_n$ specially denotes the $n$-dimensional identity matrix. Specifically, the concrete iterative scheme of Algorithm \ref{alg1} reads as
\begin{equation}\label{DB-LALM}
	\left\{ \begin{aligned}
		\tilde{y}^{k+1} &= y^k+ \frac{1}{\gamma}(A A^\top + \epsilon I)^{-1}(Ax^k-b),  \\
		x^{k+1} &= \arg\min_{x\in \mathcal{X}}\left\{ f(x) + \frac{\mu}{2}\left\|x-\left( x^k -\frac{1}{\mu}A^\top \tilde{y}^{k+1}\right)\right\|^2 \right\},   \\
		y^{k+1} &= y^k+ \frac{1}{\gamma}(AA^\top  + \epsilon I)^{-1}(Ax^{k+1}-b).
	\end{aligned}\right.
\end{equation}
It should be noted that, for the case where $n \ll m$, we can employ the well-known Sherman-Morrison-Woodbury theorem to reduce the computational cost of computing $(AA^\top  + \epsilon I)^{-1}$, i.e., $(AA^\top  + \epsilon I)^{-1}=\epsilon^{-1}\left(I_m - A(\epsilon I_ n + A^\top A )^{-1}A^\top \right)$. In what follows, we call the scheme \eqref{DB-LALM} Doubly Balanced Augmented Lagrangian Method (DBALM).  Moreover, when $f(x)$ is a quadratic function, i.e., $f(x)=\frac{1}{2} \|Bx-q\|^2$, we can also take $\psi(x)=\frac{1}{2}\|x\|^2_{(\tau I - B^\top B)}$ to derive a linearized version of DBALM to deal with the case  where $\X$ is a simple convex set whose projection is easily calculated.

\subsection{Multi-block case: Jacobian splitting method}\label{sec-multi-block}
In this part, we are concerned with the multi-block linearly constrained convex minimization, which takes the form
\begin{equation}\label{multi-block-min}
	\min_{x_1,x_2,\ldots,x_p} \left\{ \sum_{i=1}^{p}f_i(x_i) \;\Big{|}\; \sum_{i=1}^{p}A_ix_i=b,\;\;x_i\in\X_i\subseteq \R^{n_i},\;i=1,2,\ldots,p\right\},
\end{equation}
where  $f_{i}(\cdot) : \mathcal{X}_{i} \rightarrow (-\infty,+\infty]$ for $i = 1, 2,\cdots , p$ are proper closed convex functions, $A_i\in \R^{m\times n_i}$ are given matrices, and $b\in\R^m$ is a given vector.
In the past decades, the multi-block model \eqref{multi-block-min} has received much considerable attention due to its widespread applications in computer sciences and automatic control, e.g., see \cite{HYZ14,TY11}. Although such a model is also a linearly constrained optimization problem, it cannot be easily solved via the aforementioned ALM \eqref{ALM} since the linear constraints make the ALM suffer from coupled subproblems so that the separability of the objective function cannot be fully exploited in algorithmic implementation. Accordingly, a series of augmented Lagrangian-based splitting methods were developed in the optimization literature, e.g., see \cite{HYZ14,He09,HHY15,HXY16,HHY16,WHML15} and references therein.

Below, we first show that our Algorithm \ref{alg1} is applicable to solving \eqref{multi-block-min}. In particular, we can easily derive that our Algorithm \ref{alg1} is indeed the fully Jacobian splitting algorithm \cite{HXY16,WDH15} by choosing appropriate Bregman kernel functions. Moreover, we can obtain some new Jacobian splitting methods for \eqref{multi-block-min}.

First, it is clear that \eqref{multi-block-min} can be rewritten into a compact form as follows:
\begin{equation}\label{multi-compact}
	\min_{\bx} \left\{ \bbf(\bx)\;|\;\bA \bx=b,\;\;\bx \in\X\right\},
\end{equation}
where $\bx:=(x_1^\top,x_2^\top,\ldots,x_p^\top)^\top$, $\bbf(\bx):=\sum_{i=1}^{p}f_{i}(x_{i}) $, $\bA:=[A_1,A_2,\ldots,A_p]$, and $\X:=\X_1\times \X_2\times\ldots\times \X_p$. Therefore, we can reformulate \eqref{multi-compact} as the form of \eqref{one-saddle}, i.e.,
\begin{equation}\label{multi-sdp}
	\min_{\bx\in\X}\max_{y \in \R^m} \left\{\L(\bx,y) := \bbf(\bx) + \left\langle  \bA\bx,y \right\rangle -\langle b,y\rangle \right\}.
\end{equation}
Then, the specific iterative scheme of Algorithm \ref{alg1} for \eqref{one-saddle} reads as
\begin{equation}\label{PD-Mult-ALM}
	\left\{ \begin{aligned}
		\tilde{y}^{k+1} &= \arg\max_{y\in \R^m}\left\{ -\langle b,y\rangle + \langle \bA\bx^k, y\rangle - \gamma \B_\phi(y,y^k) \right\},  \\
		\bx^{k+1} &= \arg\min_{\bx\in \mathcal{X}}\left\{ \bbf(\bx) + \langle \bA\bx, \tilde{y}^{k+1} \rangle +\mu\B_\psi(\bx,\bx^k) \right\},   \\
		y^{k+1} &= \arg\max_{y\in \R^m}\left\{ -\langle b,y\rangle + \langle \bA\bx^{k+1}, y\rangle - \gamma \B_\phi(y,y^k) \right\}.
	\end{aligned}\right.
\end{equation}
Now, we denote
$$M:=\left(\begin{array}{cccc}
	\beta_1A_1^\top A_1 & 0 & \cdots & 0 \\
	0 & \beta_2A_2^\top A_2 & \cdots & 0 \\
	\vdots & \vdots &  \ddots & \vdots \\
	0 & 0 & \cdots & \beta_pA_p^\top A_p
\end{array}\right).$$
Then, by setting the Bregman kernel functions as $\phi(y)=\frac{1}{2}\|y\|^2$ and $\psi(\bx)=\frac{1}{2}\|\bx\|^2_M$ and using the separability of the objective function, for $\mu=1$, the iterative scheme \eqref{PD-Mult-ALM} reads as
\begin{equation}\label{PD-Full-Jacobian}
	\left\{ \begin{aligned}
		\tilde{y}^{k+1} &= y^k+ \frac{1}{\gamma}\left(\sum_{i=1}^{p}A_ix_i^k-b\right),  \\
		x_i^{k+1} &= \arg\min_{x_i\in \X_i}\left\{ f_i(x_i) + \frac{ \beta_i }{2}\left\|Ax_i-\left( Ax_i^k -\frac{1}{\beta_i} \tilde{y}^{k+1}\right)\right\|^2 \right\}, \; i=1,\ldots,p,  \\
		y^{k+1} &= y^k+ \frac{1}{\gamma}\left(\sum_{i=1}^{p}A_ix_i^{k+1}-b\right),
	\end{aligned}\right.
\end{equation}
which actually corresponds to the proximal Jacobian splitting method studied in \cite{HXY16}, especially precisely coincides with the augmented Lagrangian-based parallel splitting method \cite{WDH15} by setting $\beta_i=1/\gamma$ for $i=1,2\cdots,p$.

More interestingly, by setting the Bregman kernel functions as $\phi(y)=\frac{1}{2}\|y\|^2$ and $\psi(\bx)=\frac{1}{2}\|\bx\|^2$, the specific iterative scheme of Algorithm \ref{alg1} for \eqref{multi-sdp} reads as
\begin{equation}\label{PD-LFJ}
	\left\{ \begin{aligned}
		\tilde{y}^{k+1} &= y^k+ \frac{1}{\gamma}\left(\sum_{i=1}^{p}A_ix_i^k-b\right),  \\
		x_i^{k+1} &= \arg\min_{x_i\in \X_i}\left\{ f_i(x_i) + \frac{\mu}{2}\left\|x_i-\left( x_i^k -\frac{1}{\mu}A_i^\top \tilde{y}^{k+1}\right)\right\|^2 \right\}, \; i=1,\ldots,p,  \\
		y^{k+1} &= y^k+ \frac{1}{\gamma}\left(\sum_{i=1}^{p}A_ix_i^{k+1}-b\right),
	\end{aligned}\right.
\end{equation}
which is a linearized parallel splitting method for \eqref{multi-block-min}. Comparing with the methods discussed in \cite{HHX14,HYZ14,HHY15}, the variant \eqref{PD-LFJ} enjoys relatively simpler iterative scheme without correction steps. Combining the ideas of \eqref{PD-Full-Jacobian} and \eqref{PD-LFJ}, we can take
$$\widehat{M}:=\left(\begin{array}{cccccc}
	\beta_1A_1^\top A_1& \cdots & 0 & 0 & \cdots & 0 \\
	\vdots & \ddots & \vdots & \vdots & \ddots & \vdots \\
	0 & \cdots & \beta_l A_l^\top A_l & 0 & \cdots & 0\\
	0 & \cdots &0 & \beta_{l+1}I & \cdots & 0 \\
	\vdots & \ddots &  \vdots & \vdots & \ddots & \vdots \\
	0 & \cdots & 0 & 0 & \cdots & \beta_p I
\end{array}\right)$$
so that $\phi(y)=\frac{1}{2}\|y\|^2$ and $\psi(\bx)=\frac{1}{2}\|\bx\|^2_{\widehat{M}}$, thereby producing a Partially Linearized Jacobian Splitting Method (PLJSM) for \eqref{multi-block-min} with setting $\mu=1$, i.e.,
\begin{equation*}
	\left\{ \begin{aligned}
		\tilde{y}^{k+1} &= y^k+ \frac{1}{\gamma}\left(\sum_{i=1}^{p}A_ix_i^k-b\right),  \\
		x_i^{k+1} &= \arg\min_{x_i\in \X_i}\left\{ f_i(x_i) + \frac{\beta_i}{2}\left\|Ax_i-\left( Ax_i^k -\frac{1}{\beta_i} \tilde{y}^{k+1}\right)\right\|^2 \right\}, \; i=1,\ldots,l,  \\
		x_j^{k+1} &= \arg\min_{x_j\in \X_j}\left\{ f_j(x_j) + \frac{\beta_j}{2}\left\|x_j-\left( x_j^k -\frac{1}{\beta_j}A_j^\top \tilde{y}^{k+1}\right)\right\|^2 \right\}, \; j=l+1,\ldots,p,  \\
		y^{k+1} &= y^k+ \frac{1}{\gamma}\left(\sum_{i=1}^{p}A_ix_i^{k+1}-b\right),
	\end{aligned}\right.
\end{equation*}
which, to our best knowledge, is not discussed in the literature. Of course, we can also follow the spirit of the DBALM \eqref{DB-LALM} to specify $\phi(y)=\frac{1}{2}\|y\|^2_{(\sum_{i=1}^{p}A_iA_i^\top + \epsilon I)}$ and $\psi(\bx)=\frac{1}{2}\|\bx\|^2_{\widehat{M}}$ to develop a doubly balanced PLJSM for \eqref{multi-block-min}.

\section{Numerical Experiments}\label{Sec5}
In this section, we conduct the numerical performance of Algorithm \ref{alg1} (denoted by SPIDA) on some well tested problems, including the basis pursuit, RPCA, and image restoration with synthetic and real-world datasets. We also compare our Algorithm \ref{alg1} with some existing state-of-the-art primal-dual-type algorithms for the purpose of showing the numerical improvement of our Algorithm \ref{alg1}. All algorithms are implemented in {\sc Matlab} 2021a and all experiments are conducted on a 64-bit Windows personal computer with Intel(R) Core(TM) i5-12500h CPU@2.50GHz and 8GB of RAM.

\subsection{Basis pursuit}\label{sec-bp}
As discussed in Section \ref{Sec-LCM}, our Algorithm \ref{alg1} (SPIDA) is applicable to dealing with linearly constrained optimization problem \eqref{one-block-min}. In this part, we are interested in the basis pursuit problem, which can be expressed mathematically as an $\ell_1$-norm minimization problem with linear constraints, i.e.,
\begin{equation}\label{bp}
	\min_{x \in \R^n}\;\left\{\; \| x \|_{1} \;|\; Ax=b\;\right\},
\end{equation}
where $A \in \mathbb{R}^{m \times n}$ is a sample matrix and $b \in \mathbb{R}^{m}$ is a measurement vector. Such a model \eqref{bp} is a fundamental problem in compressed sensing \cite{Don06} and can be efficiently solved via a large number of optimization solvers. Here, we just employ this example to investigate the ability of our SPIDA on solving linearly constrained optimization problems, in addition to showing the superiority of SPIDA over some popular first-order optimization methods.

First, by the Lagrangian function, we reformulate \eqref{bp} as the following min-max saddle point problem:
\begin{equation*}
	\min_{x \in \R^n}\max_{y \in \R^m} \left\{\; \L(x,y)=\| x \|_{1} + \langle Ax,y \rangle -\langle b,y \rangle\; \right\}.
\end{equation*}
We compare our SPIDA with PDHG (setting $\tau=1$ in \eqref{fopda}) and GRPDA \cite{CY21}. Most recently, He and Yuan \cite{HY21} introduced a novel Balanced Augmented Lagrangian Method (BALM) for linearly convex programming, which is a great improvement of the classical ALM \eqref{ALM}. The iterative scheme of BALM for \eqref{bp} reads as
\begin{equation}\label{BALM}
	\left\{ \begin{aligned}
		x^{k+1} &= \arg\min_{x\in \R^n}\left\{ f(x) + \frac{\gamma}{2}\left\| x - \left( x^k - \frac{1}{\gamma}A^\top y^k \right)\right\|^2 \right\},   \\
		y^{k+1} &= y^k + \left(\frac{1}{\gamma}AA^\top  + \epsilon I\right)^{-1}\left(A(2x^{k+1}-x^k)-b\right).
	\end{aligned}\right.
\end{equation}
In this part, we also compare our SPIDA with the BALM \eqref{BALM}. Moreover, we will follow the idea of BALM to produce a doubly balanced augmented Lagrangian method (see \eqref{DB-LALM} and denote it by SPIDA-II) via choosing the Type II Bregman kernel function in Table \ref{Tab_Bregman}. As discussed in Section \ref{sec-one-block}, when the Bregman kernel functions are specified as the Type I of Table \ref{Tab_Bregman}, our SPIDA reduces to the LALM \eqref{LALM}, which will be denoted by SPIDA-I in our numerical comparison.

In the experiments, we first construct a randomly $s$-sparse vector $x^*\in\R^n$, where $s$ is the number of nonzero components. Then, we randomly generate a sample matrix $A\in\R^{m\times n}$ to construct the measurement vector $b$ via $b=Ax^*$. Here, we consider two different ways to generate the sample matrix $A$:
\begin{itemize}
	\item $A$ is a random Gaussian matrix;
	\item $A$ is a random partial DCT (discrete cosine transform) matrix.
\end{itemize}
We conduct different sizes of the problems by setting $(m,n,s)=(180i,960i,30i)$ with $i=1,2,\cdots,10$. Besides, to implement these algorithms, we take $(\gamma,\mu)=(1,1)$ for PDHG,  $(\varrho,\tau,\sigma)=(2,\frac{\sqrt{2}}{2},\frac{\sqrt{2}}{2})$ for GRPDA (where $\varrho$
is the parameter associated with the Golden ratio step, and $(\tau,\sigma)$ are the proximal parameters for the $x$- and $y$-subproblems, respectively), $(\gamma,\epsilon)=(1.5,0.015)$ for BALM \eqref{BALM},  $(\gamma,\mu)=(0.6,0.6)$ for SPIDA-I and $(\gamma,\mu,\epsilon)=(0.6,0.6,0.01)$ for SPIDA-II. All algorithms start with zero initial points and stop at
\begin{equation}\label{stop}
	\text{Tol}:=\frac{\|(x^{k+1},y^{k+1})-(x^k,y^k)\|}{\|(x^k,y^k)\|}\leq \varepsilon.
\end{equation}
They are terminated when satisfying the stopping criterion \eqref{stop} with $\varepsilon=10^{-6}$. The average calculation results of 10 trials of randomly generating matrix A are summarized in Tables \ref{tab3} and \ref{tab4}.

\begin{table}[!h]
	\caption{Numerical results for basis pursuit: (i) $A$ is a random Gaussian matrix.}\label{tab3}
	\centering
	\scriptsize{\begin{tabular*}{\textwidth}{@{\extracolsep{\fill}}llllllllll}\toprule
			& PDHG && GRPDA && BALM && SPIDA-I && SPIDA-II \\
			\cline{2-2} \cline{4-4} \cline{6-6} \cline{8-8} \cline{10-10}
			$i$& Iter. / Time &&  Iter. / Time  &&  Iter. / Time &&  Iter. / Time && Iter. / Time \\ \midrule
			$i=1$ & 689.5 / 1.36  && 587.5 / 1.16  && 292.6 / 0.60  && 348.3 / 0.70  && 167.6 / 0.41   \\
			$i=2$ & 1036.8 / 6.36  && 811.2 / 4.98  && 374.7 / 2.33  && 572.9 / 3.57  && 257.4 / 1.68   \\
			$i=3$ & 1525.3 / 16.20  && 1135.2 / 12.06  && 518.8 / 5.58  && 875.2 / 9.42  && 393.5 / 6.54   \\
			$i=4$ & 1595.8 / 27.53  && 1182.8 / 20.39  && 536.6 / 9.38  && 919.4 / 16.02  && 410.7 / 10.73   \\
			$i=5$ & 1295.3 / 31.37  && 980.6 / 23.78  && 451.2 / 11.19  && 737.3 / 18.03  && 335.9 / 11.96   \\
			$i=6$ & 3272.4 / 113.45  && 2346.8 / 81.19  && 1037.2 / 36.20  && 1938.0 / 67.91  && 887.9 / 42.56   \\
			$i=7$ & 1038.6 / 43.53  && 794.8 / 33.35  && 380.3 / 16.12  && 583.8 / 24.82  && 275.2 / 15.78   \\
			$i=8$ & 1567.3 / 85.37  && 1161.6 / 63.31  && 538.2 / 29.68  && 906.2 / 50.64  && 415.7 / 30.49   \\
			$i=9$ & 7314.6 / 472.55  && 5177.1 / 334.80  && 2389.7 / 155.92  && 4379.6 / 290.47  && 2073.1 / 179.49   \\
			$i=10$ & 8750.3 / 740.80  && 6198.3 / 527.87  && 2667.6 / 227.76  && 5238.9 / 450.59  && 2299.4 / 250.56   \\
			\bottomrule
	\end{tabular*}}
\end{table}

\begin{table}[!h]
	\caption{Numerical results for basis pursuit: (ii) $A$ is a random partial DCT matrix.}\label{tab4}
	\centering
	\scriptsize{\begin{tabular*}{\textwidth}{@{\extracolsep{\fill}}llllllllll}\toprule
			& PDHG && GRPDA && BALM && SPIDA-I && SPIDA-II \\
			\cline{2-2} \cline{4-4} \cline{6-6} \cline{8-8} \cline{10-10}
			$i$& Iter. / Time &&  Iter. / Time  &&  Iter. / Time &&  Iter. / Time && Iter. / Time \\ \midrule
			$i=1$ & 291.2 / 0.57  && 253.8 / 0.50  && 255.5 / 0.52  && 142.0 / 0.29  && 143.8 / 0.36  \\
			$i=2$ & 1278.4 / 8.21  && 901.5 / 5.77  && 873.4 / 5.63  && 756.0 / 4.95  && 763.8 / 5.46  \\
			$i=3$ & 397.0 / 4.61  && 325.0 / 3.76  && 322.3 / 3.77  && 207.9 / 2.43  && 210.1 / 3.79  \\
			$i=4$ & 1207.0 / 22.70  && 873.6 / 16.45  && 846.5 / 16.07  && 707.6 / 13.40  && 714.8 / 18.49  \\
			$i=5$ & 366.1 / 9.10  && 304.2 / 7.59  && 304.1 / 7.63  && 189.3 / 4.79  && 191.2 / 6.71  \\
			$i=6$ & 2558.8 / 91.87  && 1843.5 / 67.13  && 1781.4 / 65.38  && 1523.1 / 56.47  && 1538.3 / 73.30  \\
			$i=7$ & 1317.5 / 60.81  && 965.5 / 44.67  && 936.3 / 43.68  && 772.1 / 36.32  && 779.8 / 46.91  \\
			$i=8$ & 3931.9 / 231.36  && 2757.4 / 161.90  && 2658.5 / 157.81  && 2353.4 / 142.21  && 2376.7 / 181.08  \\
			$i=9$ & 631.7 / 45.56  && 482.6 / 34.82  && 471.0 / 34.29  && 357.9 / 26.63  && 361.6 / 33.40  \\
			$i=10$ & 1001.0 / 87.46  && 728.6 / 63.62  && 707.4 / 62.20  && 584.1 / 52.57  && 589.9 / 65.05  \\
			\bottomrule
	\end{tabular*}}
\end{table}

It can be easily seen from Table \ref{tab3} that our SPIDA-II takes the fewest iterations to obtain approximate solutions for the case where $A$ is a random Gaussian matrix. When dealing with the other case where $A$ is a partial DCT matrix, results in Table \ref{tab4} tell us that SPIDA-I and SPIDA-II have the almost same performance, while taking fewer iterations to achieve high-quality solutions than the other three first-order algorithms for \eqref{bp}. These computational results demonstrate that the symmetric updating way on the dual variable (i.e., twice calculations) equipped with a general proximal regularization can improve the numerical performance of the classical ALM. To further show the convergence behavior of our SPIDA, we focus on the case with $(m,n,s)=(360,1920,60)$ and plot the convergence curve of the relative error defined by ${\| x^{k} - x^{*} \|}/{\| x^{*} \|}$ with respect to iterations in Fig. \ref{fig-bp}. We see from Fig. \ref{fig-bp} that our SPIDA has a promisingly linear convergence behavior for one-block linearly constrained optimization problems.

\begin{figure}[!h]
	\centering
	\subfigure[Gaussian matrix]{
		\label{fig-bp-1}
		\includegraphics[width=0.47\textwidth]{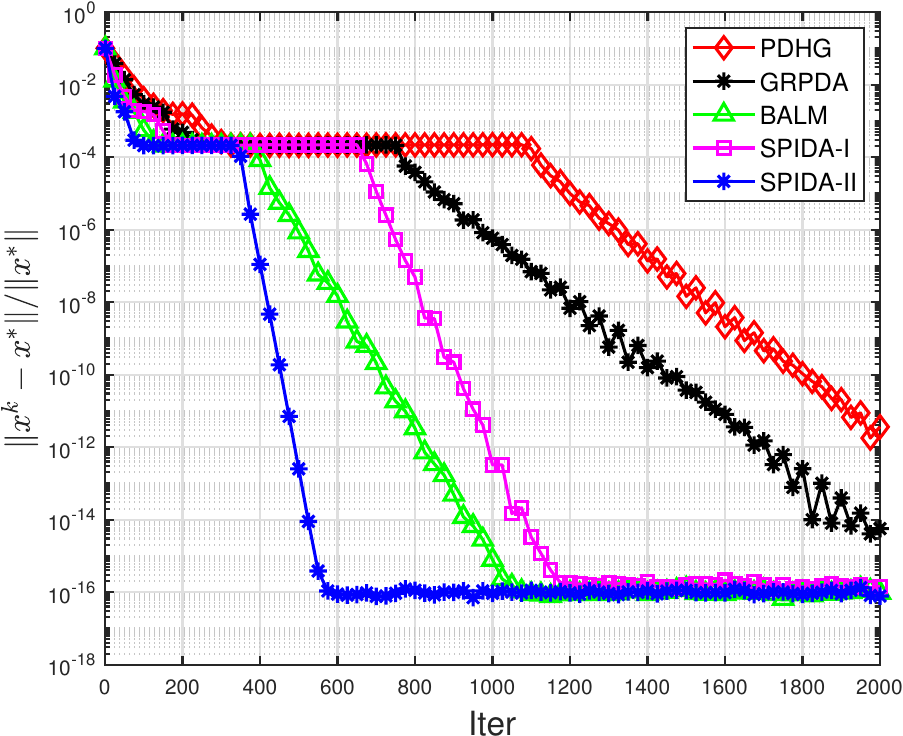}}
	\subfigure[partial DCT matrix]{
		\label{fig-bp-2}
		\includegraphics[width=0.47\textwidth]{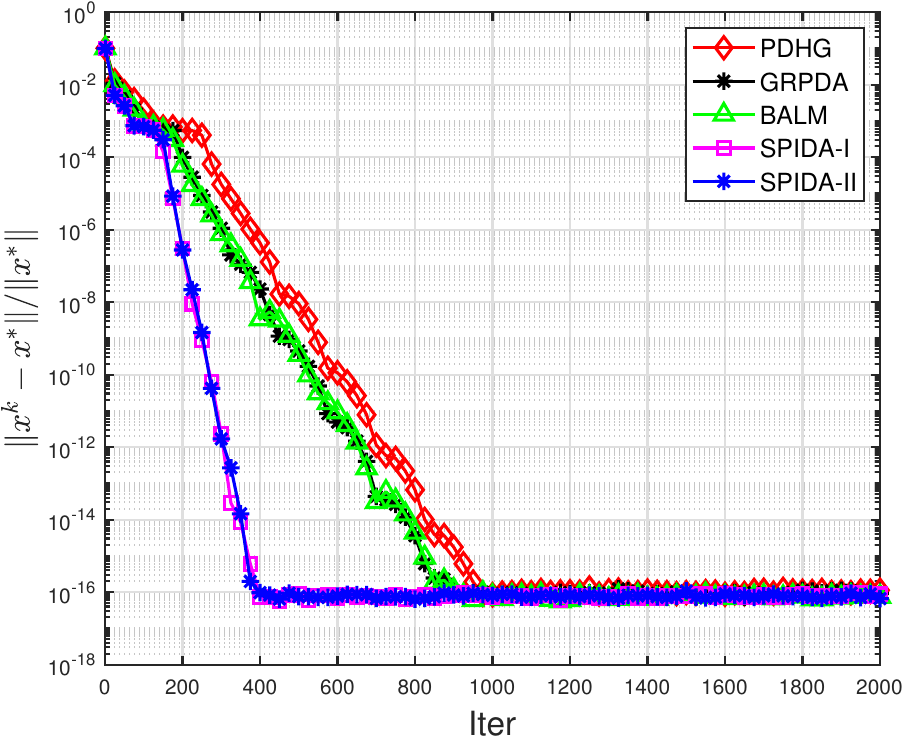}}
	\caption{Evolution of the relative error with respect to iterations for solving the basis pursuit problem \eqref{bp}.}
	\label{fig-bp}
\end{figure}

Due to the randomness of the generated data sets for basis pursuit, we are further interested in the stability of our SPIDA in practice. Therefore, we show the averaged time and iterations of ten trials by the bars, and their standard deviations by the line segments for the cases $i=9$ and $i=10$ in Fig. \ref{bpl}. Comparatively, it is illustrated in Fig. \ref{bpl} that our SPIDA performs stably for the random data sets. 

\begin{figure}[!h]
	\centering
	\subfigure[Gaussian matrix]{
		\includegraphics[width=0.478\textwidth]{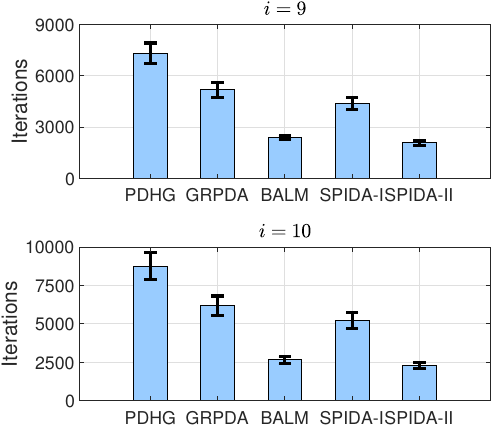}
		\includegraphics[width=0.472\textwidth]{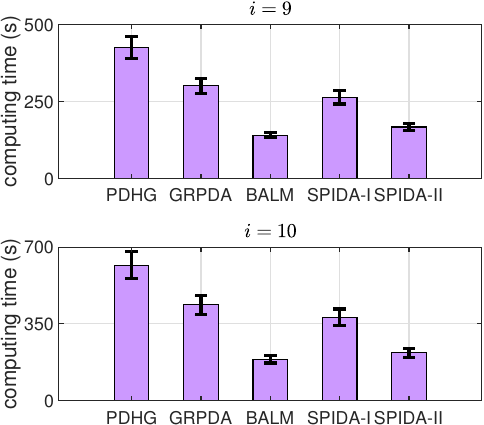}
	}
	\subfigure[partial DCT matrix]{
		\includegraphics[width=0.482\textwidth]{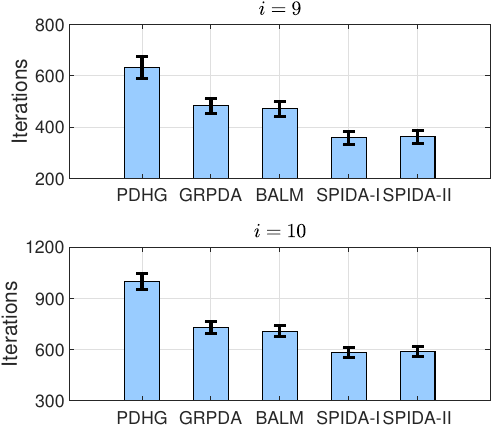}
		\includegraphics[width=0.468\textwidth]{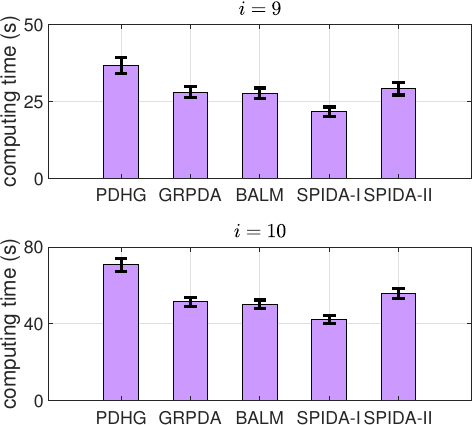}
	}
	\caption{Stability investigation of all algorithms for solving basis pursuit by setting $i=9$ and $i=10$, where the averaged iterations and computing time are shown by the bars, and their standard deviations is plotted by the line segments.}
	\label{bpl}
\end{figure}

\subsection{RPCA}\label{sec-rpca}
In this subsection, we give a numerical feedback to the applicability of our SPIDA to multi-block convex programming discussed in Section \ref{sec-multi-block}. Therefore, we consider a well-studied RPCA model \cite{CLMW11}, which refers to the task of recovering a sparse matrix and a low-rank one. Specifically, the RPCA under consideration takes the form
\begin{equation}\label{P1}
	\min_{X,Z}\; \left\{\;\| X \|_{*} + \lambda \| Z \|_{1} \; |\; X+Z=H \; \right\},
\end{equation}
where $\| X \|_{*}$ is the nuclear norm (i.e., the sum of all singular values of $X$) for promoting the low-rankness of $X\in\R^{m\times n}$,  $\| Z \|_{1} $ represents the $\ell_{1}$-norm for inducing the sparsity of $Z\in\R^{m\times n}$, and $H\in\R^{m\times n}$ is a given matrix. Obviously, the RPCA model \eqref{P1} is a special case of the multi-block model \eqref{multi-block-min} with two-block structure. Therefore, we can easily reformulate \eqref{P1} as the following separable saddle point problem:
\begin{equation*}\label{P2}
	\min_{X,Z}\max_{Y}\; \left\{\;\L(X,Z,Y)=\| X \|_{*} + \lambda \| Z \|_{1} + \langle X+Z,Y \rangle - \langle H,Y \rangle\;\right\}.
\end{equation*}
In this part, we also mainly compare SPIDA with PDHG and GRPDA. Here, we consider the RPCA with synthetic and real data sets to verify the reliability of our algorithm for multi-block convex programming.

We first conduct the numerical performance of these algorithms on synthetic data sets. In this situation, we generate a low-rank matrix $X^*$ via $X^* = UV$, where $U\in\R^{n\times r}$ and $V\in\R^{r\times n}$ are independently random matrices whose entries are drawn from Gaussian distribution $\mathcal{N}(0,1)$. Then, we generate a sparse matrix $Z^*$ by randomly choosing a support set $\Omega$ of size $0.1\times n^2$ (i.e., $10\%$ nonzero components), and all elements are independently sampled from a uniform distribution in $[-50,50]$. Finally, we let $H = X^*+Z^*$ be the observed matrix. Clearly, $(X^*,Z^*)$ is the true solution of \eqref{P1}. Throughout our experiments, we still employ the stopping criterion \eqref{stop} with setting $\varepsilon= 10^{-5}$ for all algorithms. Besides, we take $(\gamma,\mu)=(70.7107,0.0283)$ for PDHG, $(\psi,\tau,\sigma)=(1.618,\frac{70.7107}{\sqrt{1.618}},\frac{0.0283}{\sqrt{1.618}})$ for GRPDA, $(\gamma,\mu)=(0.77*70.7107,0.0283)$ for SPIDA. In Table \ref{tab5}, we additionally report the rank of the obtained low-rank matrix ($\text{rank}(\hat{X})$), the number of nonzero components of the obtained sparse matrix ($\|\hat{Z}\|_0$), the relative error ({\sffamily{Rerr}}) defined by
\begin{equation*}
	\text{\sffamily Rerr} = \frac{\| \hat{X} + \hat{Z} -X^* - Z^*\|_{F}}{\| X^*+Z^*\|_{F}},
\end{equation*}
where $\hat{X}$ and $\hat{Z}$ respectively represent the low-rank and sparse matrices obtained by the algorithms. It can be seen from Table \ref{tab5} that our SPIDA takes less iterations and computing time than both PDHG and GRPDA to achieve almost the same low-rank and sparse separation on the observed matrix $H$.

\begin{table}[htbp]
	\caption{Numerical results of RPCA with synthetic data sets.}\label{tab5}
	\centering
	\small{\begin{tabular*}{\textwidth}{@{\extracolsep{\fill}}ccccccc}
			\toprule
			$(n,r)$ & Methods & $\text{rank}(\hat{X})$ & $ \| \hat{Z} \|_{0} $ & $ \text{\sf Rerr}$ & Iter. & Time \\ \toprule
			\multirow{3}{*}{\shortstack{  $(256,13)$}}
			& PDHG & 13 & 6528 & 6.3011$\times 10^{-4}$ & 146 & 1.13 \\
			& GRPDA & 13 & 6518 & 4.6516$\times 10^{-4}$ & 138 & 1.06 \\
			& SPIDA & 13 & 6523 & 6.3007$\times 10^{-4}$ & 112 & 0.82 \\
			\midrule
			\multirow{3}{*}{\shortstack{  $(512,26)$}}
			& PDHG & 26 & 26143 & 1.2775$\times 10^{-4}$ & 127 & 4.73 \\
			& GRPDA & 26 & 26151 & 1.1530$\times 10^{-4}$ & 119 & 4.50 \\
			& SPIDA & 26 & 26128 & 1.7869$\times 10^{-4}$ & 86 & 3.23 \\
			\midrule
			\multirow{3}{*}{\shortstack{  $(1024,51)$}}
			& PDHG & 51 & 104677 & 7.0827$\times 10^{-5}$ & 78 & 14.32 \\
			& GRPDA & 51 & 104725 & 5.0445$\times 10^{-5}$ & 87 & 15.48 \\
			& SPIDA & 51 & 104677 & 7.0421$\times 10^{-5}$ & 61 & 10.88 \\
			\midrule
			\multirow{3}{*}{\shortstack{  $(2048,102)$}}
			& PDHG & 102 & 419106 & 2.4605$\times 10^{-5}$ & 80 & 269.52 \\
			& GRPDA & 102 & 419221 & 1.7315$\times 10^{-5}$ & 105 & 378.50 \\
			& SPIDA & 102 & 419106 & 2.4006$\times 10^{-5}$ & 67 & 205.77 \\
			\midrule
			\multirow{3}{*}{\shortstack{  $(2560,128)$}}
			& PDHG & 128 & 654970 & 1.6256$\times 10^{-5}$ & 93 & 299.85 \\
			& GRPDA & 128 & 655097 & 1.5364$\times 10^{-5}$ & 124 & 383.87 \\
			& SPIDA & 128 & 654965 & 1.5685$\times 10^{-5}$ & 80 & 248.64 \\
			\bottomrule
	\end{tabular*}}
\end{table}

Below, we are concerned with the numerical performance of SPIDA on RPCA with real data sets. So, we consider the application of model \eqref{P1} in background separation of surveillance video. Here, we select three well-tested videos, i.e., {\sf Shoppingmall,} {\sf Lobby}, and {\sf Hall Airport}, and select the first $200$ frames of each video to construct an observed matrix $H \in \mathbb{R}^{n \times 200}$, where $n=n_1\times n_2$ with $n_1$ and $n_2$ representing the height and width of the video, respectively. Notice that the true rank and sparsity of these videos are unknown. Therefore, we shall report the number of iterations, the computing time in seconds, the objective values ({\sf Obj.}), and the error ({\sf Err.}) defined by
\begin{equation*}
	\text{\sf Err} = \frac{\| \hat{X} + \hat{Z} - H\|_{F}}{\| H \|_{F}}.
\end{equation*}
Throughout, we set $\varepsilon=5 \times 10^{-4}$ in \eqref{stop} as the stopping tolerance for all algorithms. Computational results are summarized in Table \ref{tab6}, which also demonstrate that our SPIDA runs a little faster than both PDHG and GRPDA for real-world data sets. In Fig. \ref{Fig5}, we list the separated background and foreground of some frames. We can see from these results that all primal-dual-type algorithms are reliable for multi-block convex programming \eqref{multi-block-min}, especially for RPCA.

\begin{table}[htbp]
	\caption{Numerical results of RPCA with real-world data sets.}\label{tab6}
	\centering
	{\begin{tabular*}{\textwidth}{@{\extracolsep{\fill}}ccccccc}
			\toprule
			$(m,n)$ & Methods & {\sf Obj.}& {\sf Err.} & Iter. & Time \\ \toprule
			\multirow{4}{*}{\shortstack{ Shoppingmall \\  $(81920,200)$}}
			& PDHG & 3262.1 & 3.0359$\times 10^{-3}$ & 58 & 33.72 \\
			& GRPDA & 3275.2 & 2.3097$\times 10^{-3}$ & 60 & 30.15 \\
			& SPIDA & 3267.7 & 2.6796$\times 10^{-3}$ & 49 & 28.52 \\
			\midrule
			\multirow{4}{*}{\shortstack{ Lobby \\  $(20480,200)$}}
			& PDHG & 969.65 & 4.1045$\times 10^{-3}$ & 128 & 13.53 \\
			& GRPDA & 977.38 & 2.9694$\times 10^{-3}$ & 127 & 13.66 \\
			& SPIDA & 973.61 & 3.5057$\times 10^{-3}$ & 109 &11.58 \\
			\midrule
			\multirow{4}{*}{\shortstack{ Hall airport \\  $(25344,200)$}}
			& PDHG & 2178.5 & 3.6816$\times 10^{-3}$ & 66 & 10.34 \\
			& GRPDA & 2188.3 & 2.9182$\times 10^{-3}$ & 68 & 8.98 \\
			& SPIDA & 2183.2 & 3.2527$\times 10^{-3}$ & 57 & 7.50 \\
			\bottomrule
	\end{tabular*}}
\end{table}

\begin{figure}[htbp]
	\begin{center}
		\includegraphics[width=0.13\linewidth]{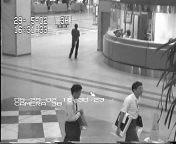}
		\includegraphics[width=0.13\linewidth]{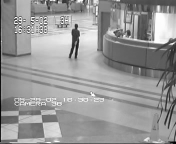}
		\includegraphics[width=0.13\linewidth]{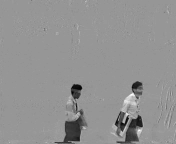}
		\includegraphics[width=0.13\linewidth]{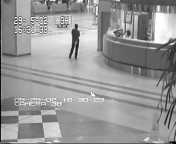}
		\includegraphics[width=0.13\linewidth]{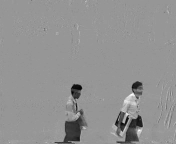}		
		\includegraphics[width=0.13\linewidth]{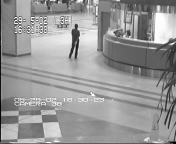}
		\includegraphics[width=0.13\linewidth]{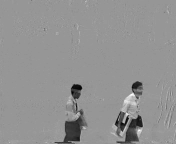}\\
		\includegraphics[width=0.13\linewidth]{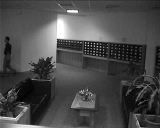}
		\includegraphics[width=0.13\linewidth]{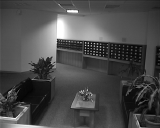}
		\includegraphics[width=0.13\linewidth]{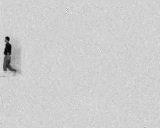}
		\includegraphics[width=0.13\linewidth]{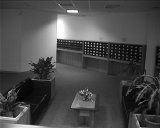}
		\includegraphics[width=0.13\linewidth]{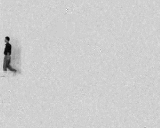}
		\includegraphics[width=0.13\linewidth]{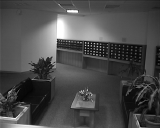}
		\includegraphics[width=0.13\linewidth]{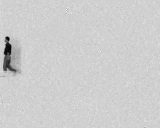}\\
		\includegraphics[width=0.13\linewidth]{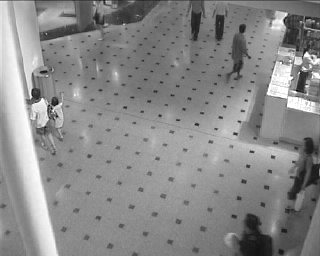}
		\includegraphics[width=0.13\linewidth]{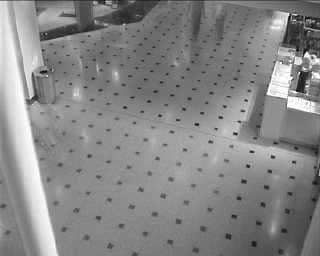}
		\includegraphics[width=0.13\linewidth]{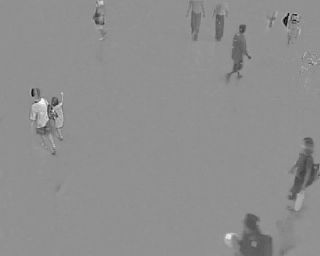}
		\includegraphics[width=0.13\linewidth]{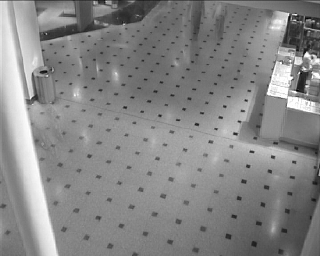}
		\includegraphics[width=0.13\linewidth]{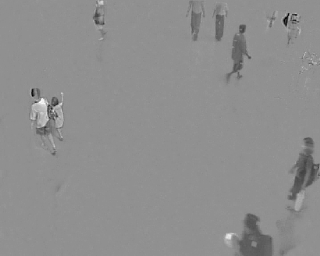}
		\includegraphics[width=0.13\linewidth]{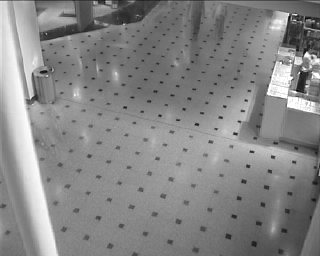}
		\includegraphics[width=0.13\linewidth]{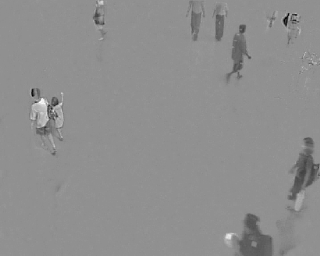}
	\end{center}
	\caption{Background and foreground separations of surveillance videos. Three lines of pictures from top to bottom correspond to the Hall airport, Lobby and Shoppingmall in the video. The first column corresponds to the 100th frame of the video. The second and third columns are the results extracted by PDHG. The fourth and fifth columns are the results extracted by PDHG.The last two columns are the results obtained by our SPIDA.}\label{Fig5}
\end{figure}

\subsection{Image restoration}\label{sec_image}
In this subsection, we apply the proposed SPIDA to solve an image restoration model with pixels constraints introduced in \cite{HHYY14}. Such a model takes the form
\begin{equation}\label{tv-l2}
	\min_{ x\in\mathbb{B}}\;\left\{ \| |{\bm D}x |\|_1 +\frac{\lambda}{2}\|Kx-b\|^2\; \right\},
\end{equation}
where $\mathbb{B}$ is a box area characterizing the pixels of an image (indeed, $\mathbb{B}=[0,1]$ and $\mathbb{B}=[0,255]$ if the image are double precision and 8-bit gray scale, respectively); $\lambda$ is a positive trade-off parameter between the data-fidelity and regularization terms; $\bm{D}:=(\partial_1,\partial_2)$ denotes the gradient operator with $\partial_1$ and $\partial_2$ being the discretized derivatives in the horizontal and vertical directions, respectively; $K$ is the matrix representation of a blur operator and $b$ is a corrupted image with additive noise. Clearly, model \eqref{tv-l2} is equivalent to the following saddle point problem:
\begin{equation}\label{tv-saddle}
	\min_{ x\in\mathbb{B}}\max_{y\in\mathbb{B}_\infty}\;\left\{ \langle {\bm D}x,y\rangle  +\frac{\lambda}{2}\|Kx-b\|^2\; \right\},
\end{equation}
where $\mathbb{B}_{\infty}:=\left\{y\;|\; \|y\|_\infty \leq 1\right\}$. Consequently, when applying PDHG \eqref{fopda} to \eqref{tv-saddle}, the iterative scheme is specified as
\begin{equation}\label{fopda-tv}
	\left\{\begin{aligned}
		y^{k+1}&=\arg\max_{y\in\mathbb{B}_{\infty}}\left\{\langle {\bm D}x^k,y\rangle -\frac{\gamma}{2}\|y-y^k\|^2\right\}\equiv \Pi_{\mathbb{B}_\infty}(y^k+\gamma^{-1}\bm{D}x^k), \\
		x^{k+1}&=\arg\min_{x\in\mathbb{B}}\left\{\frac{\lambda}{2}\|Kx-b\|^2 + \langle {\bm D}x,2y^{k+1}-y^k\rangle + \frac{\mu}{2}\|x-x^k\|^2\right\},
	\end{aligned}\right.
\end{equation}
where the updating order of $x$ and $y$ is exchanged since the $y$-subproblem is simpler than the $x$-part, while the appearance of the deblurring matrix $K$ makes $x$-subproblem relatively difficult without a closed-form solution. In this case, we employ the projected Barzilai-Borwein method in \cite{DF05} and allow a maximal number of $50$ for the inner loop to find an approximate solution of the $x$-subproblem in \eqref{fopda-tv}. To apply our SPIDA to \eqref{tv-saddle}, we consider two choices on the Bregman kernel functions: (i) $\psi(x)=\frac{1}{2}\|x\|^2$ and $\phi(y)=\frac{1}{2}\|y\|^2$; (ii) $\psi(x)=\frac{1}{2}\|x\|^2_{(I -\lambda\mu^{-1}K^\top K)}$ and $\phi(y)=\frac{1}{2}\|y\|^2$. In what follows, we denote our SPIDA equipped with the above two kernel functions by SPIDA-I and SPIDA-II, respectively. As a consequence, SPIDA-I and SPIDA-II are specified as
\begin{equation*}
	\text{(SPIDA-I)}\quad \left\{\begin{aligned}
		\tilde{y}^{k+1}&= \Pi_{\mathbb{B}_\infty}(y^k+\gamma^{-1}\bm{D}x^k), \\
		x^{k+1}&=\arg\min_{x\in\mathbb{B}}\left\{\frac{\lambda}{2}\|Kx-b\|^2 + \langle {\bm D}x,\tilde{y}^{k+1}\rangle + \frac{\mu}{2}\|x-x^k\|^2\right\},\\
		y^{k+1}&=\Pi_{\mathbb{B}_\infty}(y^k+\gamma^{-1}\bm{D}x^{k+1}),
	\end{aligned}\right.
\end{equation*}
and
\begin{equation*}
	\text{(SPIDA-II)}\quad\left\{\begin{aligned}
		\tilde{y}^{k+1}&= \Pi_{\mathbb{B}_\infty}(y^k+\gamma^{-1}\bm{D}x^k), \\
		x^{k+1}&=\Pi_{\mathbb{B}}\left( x^k-\lambda\mu^{-1}K^\top(Kx^k-b)-\mu^{-1}\bm{D}^\top \tilde{y}^{k+1}\right),\\
		y^{k+1}&=\Pi_{\mathbb{B}_\infty}(y^k+\gamma^{-1}\bm{D}x^{k+1}).
	\end{aligned}\right.
\end{equation*}
We also employ the projected Barzilai-Borwein method to find an approximate solution of the $x$-subproblem of SPIDA-I.

\begin{figure}[htbp]
	\centering
	\includegraphics[width=0.16\linewidth]{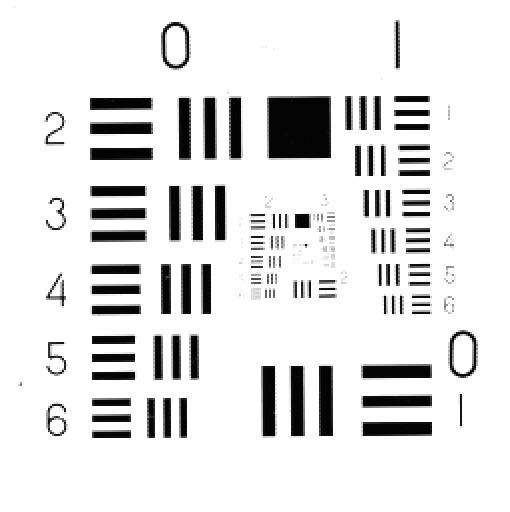}
	\includegraphics[width=0.16\linewidth]{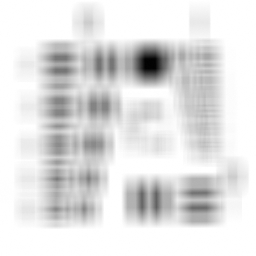}
	\includegraphics[width=0.16\linewidth]{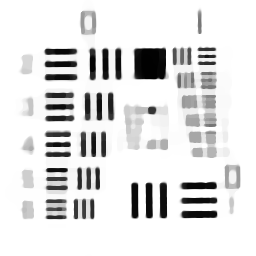}
	\includegraphics[width=0.16\linewidth]{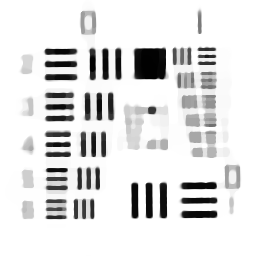}
	\includegraphics[width=0.16\linewidth]{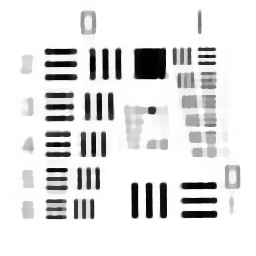}
	\includegraphics[width=0.16\linewidth]{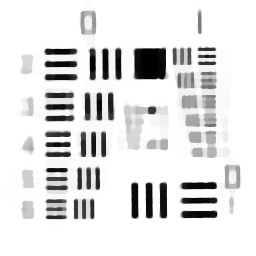}
	\\
	\includegraphics[width=0.16\linewidth]{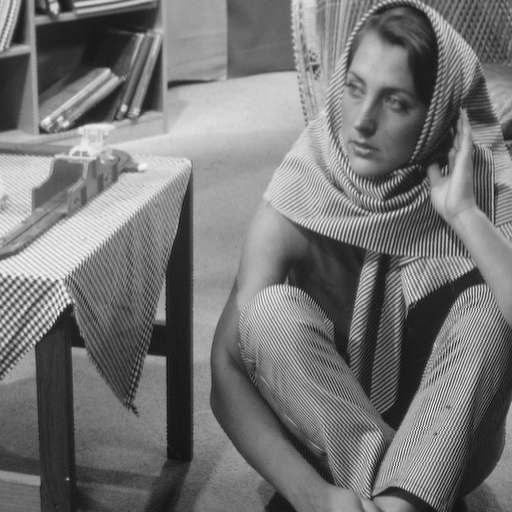}
	\includegraphics[width=0.16\linewidth]{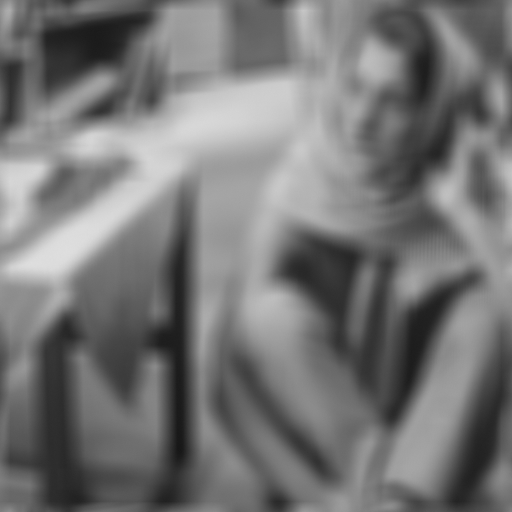}
	\includegraphics[width=0.16\linewidth]{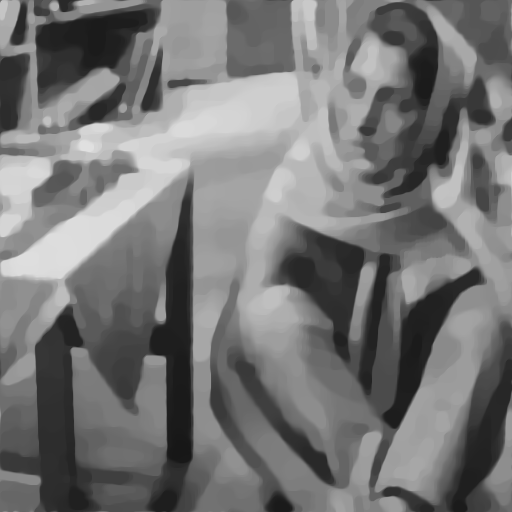}	
	\includegraphics[width=0.16\linewidth]{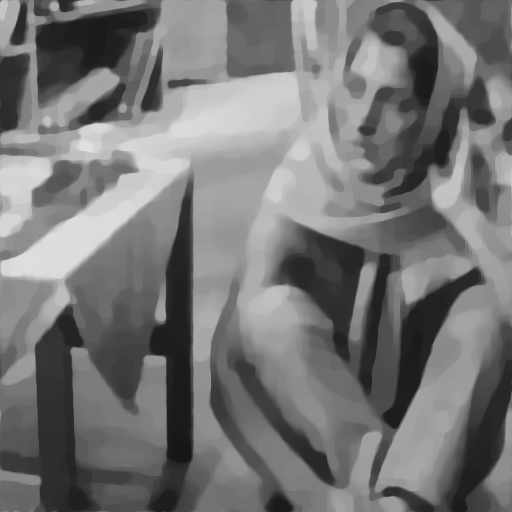}
	\includegraphics[width=0.16\linewidth]{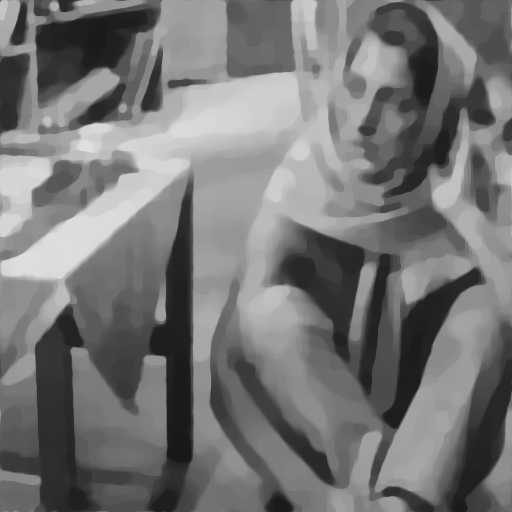}
	\includegraphics[width=0.16\linewidth]{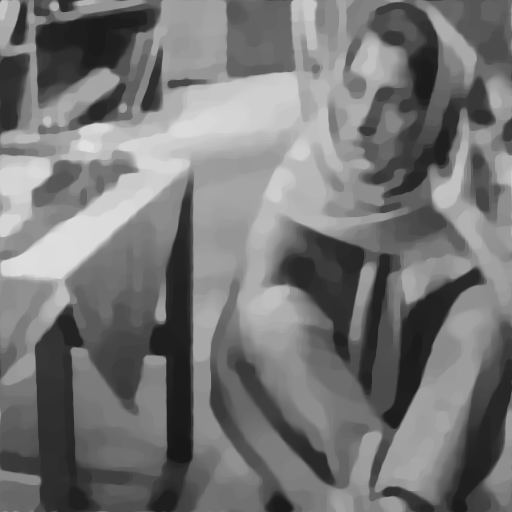}
	\\
	\includegraphics[width=0.16\linewidth]{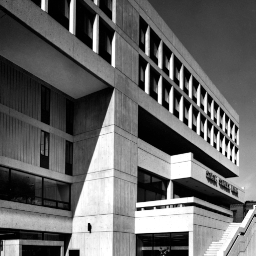}
	\includegraphics[width=0.16\linewidth]{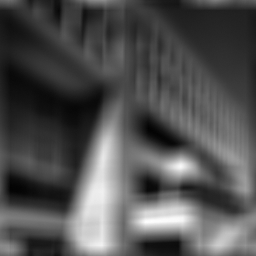}
	\includegraphics[width=0.16\linewidth]{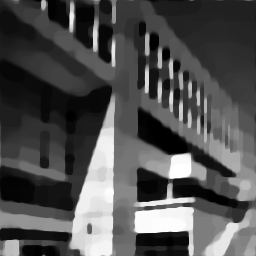}	
	\includegraphics[width=0.16\linewidth]{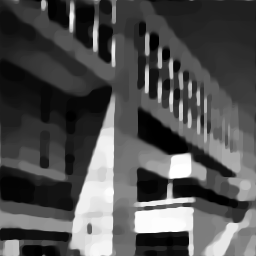}
	\includegraphics[width=0.16\linewidth]{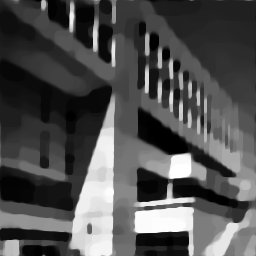}
	\includegraphics[width=0.16\linewidth]{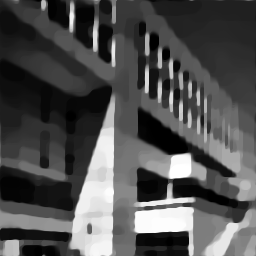}
	\\
	\includegraphics[width=0.16\linewidth]{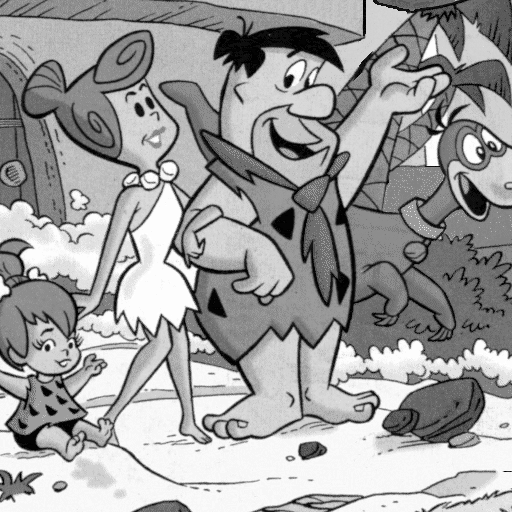}
	\includegraphics[width=0.16\linewidth]{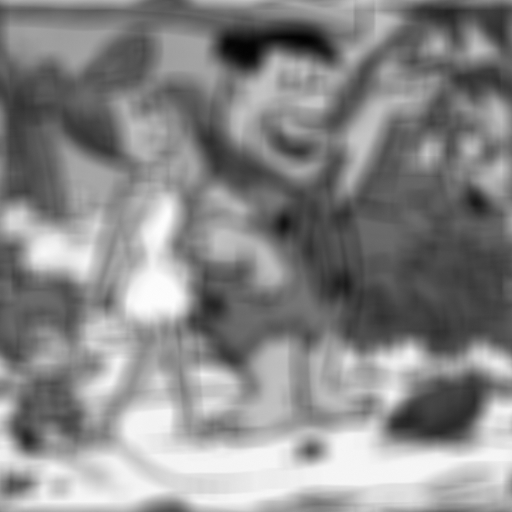}
	\includegraphics[width=0.16\linewidth]{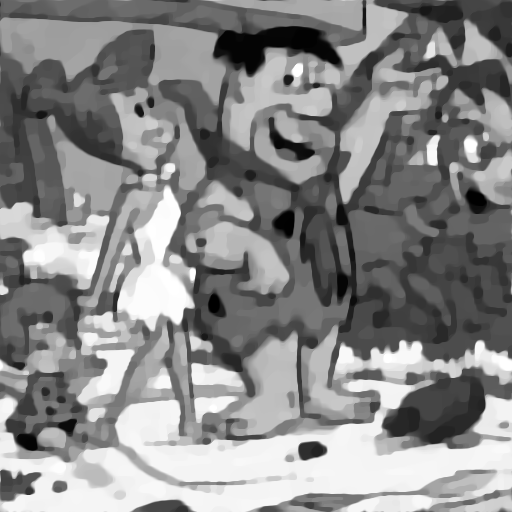}	
	\includegraphics[width=0.16\linewidth]{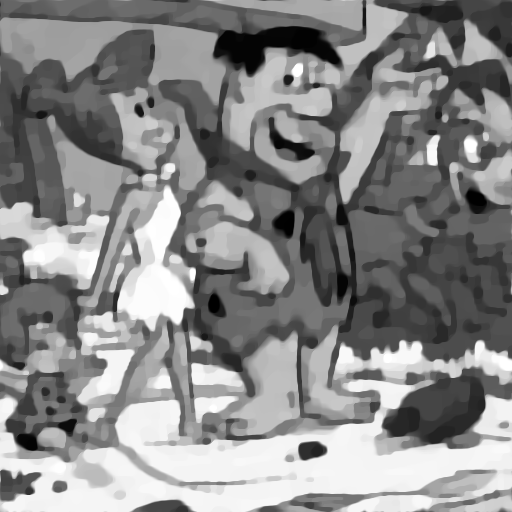}
	\includegraphics[width=0.16\linewidth]{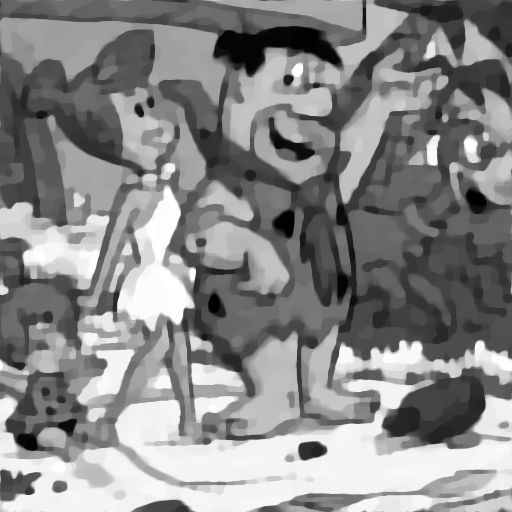}
	\includegraphics[width=0.16\linewidth]{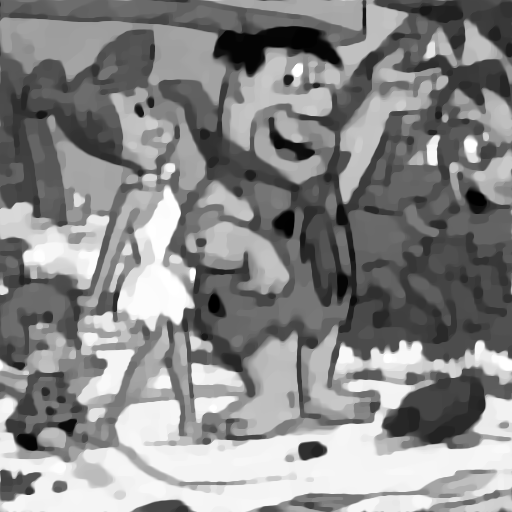}	
	\caption{From left to right: the original image, observed image corrupted by blurring kernel size $21 \times 21$, recovered images by PDHG, GRPDA, SPIDA-I and SPIDA-II, respectively.}
	\label{Fig6}
\end{figure}

In our experiments, we consider several widely tested images as listed in the first column of Fig. \ref{Fig6}. These images are corrupted by the blur operator with $21\times 21$ uniform kernels. Then, the blurred images are further corrupted by adding the zero-white Gaussian noise with standard deviation $0.002$. The degraded images are listed in the second column of Fig. \ref{Fig6}. We set $\varepsilon= 10^{-4}$ in \eqref{tv-saddle} as the stopping tolerance for all algorithms. The trade-off parameter $\lambda$ is specified as $\lambda=1000$. Moreover, we take $(\gamma,\mu)=(0.016,500)$ for PDHG, $(\psi,\tau,\sigma)=(1.618,0.016,\frac{500}{1.618})$ for GRPDA, $(\gamma,\mu)=(0.016,0.75*500)$ for SPIDA-I and $(\gamma,\mu)=(0.016,500)$ for SPIDA-II. Some preliminary numerical results are summarized in Table \ref{tab7}, which clearly shows that our SPIDA-I and SPIDA-II perform better than PDHG and GRPDA in terms of iterations and computing time. Note that SPIDA-I requires less iterations to achieve the almost same SNR values than SPIDA-II. However, SPIDA-II takes much less computing time than SPIDA-I, thanks to the closed-form solutions of SPIDA-II. These results efficiently support that our symmetric idea is able to improve the numerical performance of the original primal-dual algorithm.

\begin{table}[!h]
	\caption{Numerical results for image restoration.}\label{tab7}
	\centering
	\scriptsize{\begin{tabular*}{\textwidth}{@{\extracolsep{\fill}}llllllll}\toprule
			& PDHG && GRPDA && SPIDA-I && SPIDA-II \\
			\cline{2-2} \cline{4-4} \cline{6-6} \cline{8-8}
			image & Iter. / Time / SNR &&  Iter. / Time / SNR  &&   Iter. / Time / SNR  && Iter. / Time / SNR \\ \midrule
			\multirow{1}{*}{\shortstack{chart}}
			& 1476 / 16.37 / 19.214 && 1890 / 25.55 / 19.187 && 1217 / 15.12 / 19.220 && 1470 / 5.00 / 19.213 \\
			\midrule
			\multirow{1}{*}{\shortstack{barbara}}
			& 793 / 37.37 / 17.032 && 1035 / 59.67 / 17.036 && 677 / 36.64 / 17.029 && 792 / 15.73 / 17.032 \\
			\midrule
			\multirow{1}{*}{\shortstack{mit}}
			& 1080 / 13.11 / 12.960 && 1433 / 21.94 / 12.942 && 886 / 12.22 / 12.965 && 1078 / 3.81 / 12.960 \\
			\midrule
			\multirow{1}{*}{\shortstack{flinstones }}
			& 1141 / 54.02 / 14.439 && 1509 / 83.65 / 14.431 && 937 / 49.62 / 14.441 && 1171 / 23.47 / 14.440 \\
			\bottomrule
	\end{tabular*}}
\end{table}

In Fig. \ref{Fig6}, we list the recovered images by PDHG, GRPDA, SPIDA-I, and SPIDA-II from the third column to the last one, respectively. It can be seen that all methods are reliable to solve model \eqref{tv-saddle}. Finally, we list the evolution of SNR value with respect to computing time for solving image restoration model \eqref{tv-saddle} in Fig. \ref{fig7}. These curves further show our SPIDA-II with easy subproblems has a superiority over the other algorithms in terms of computing time.

\begin{figure}[!h]
	\centering
	\subfigure[chart]{
		\includegraphics[width=0.485\textwidth,height=0.3\textheight]{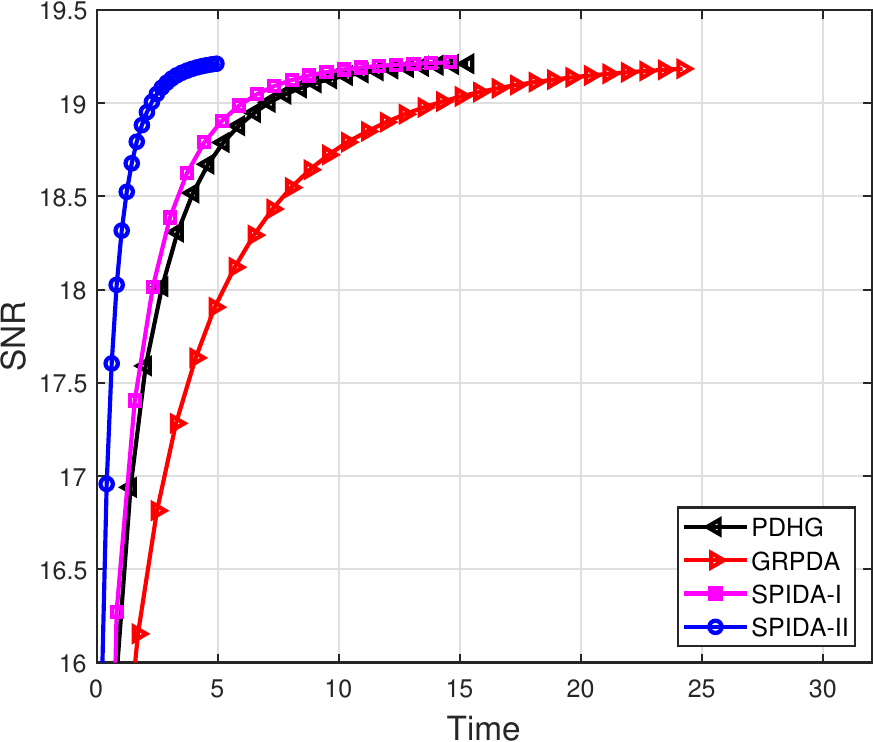}}
	\subfigure[barbara]{
		\includegraphics[width=0.485\textwidth,height=0.3\textheight]{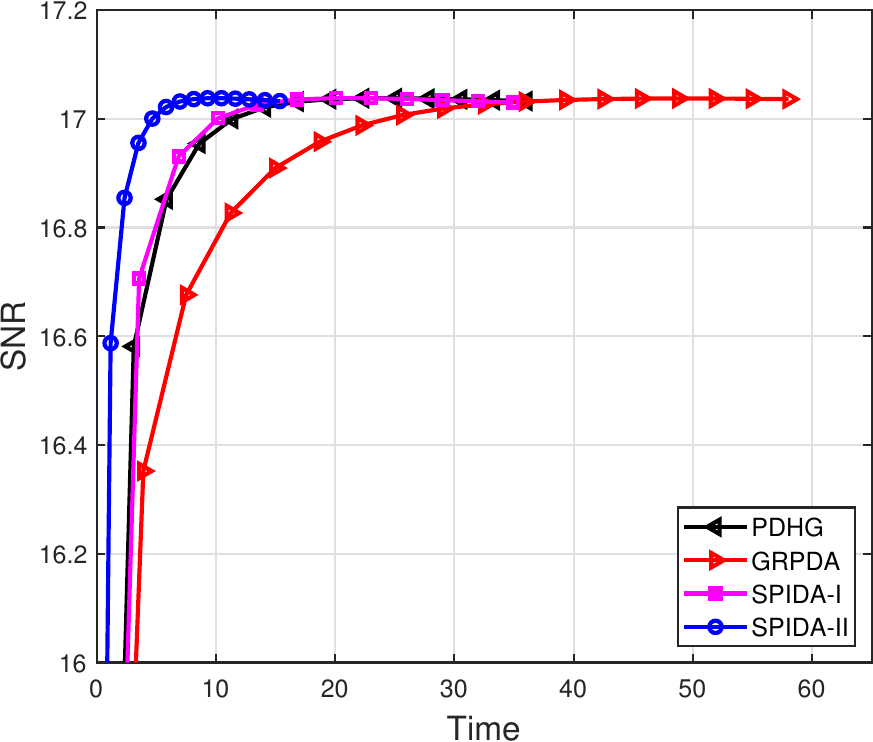}}
	\subfigure[mit]{
		\includegraphics[width=0.485\textwidth,height=0.3\textheight]{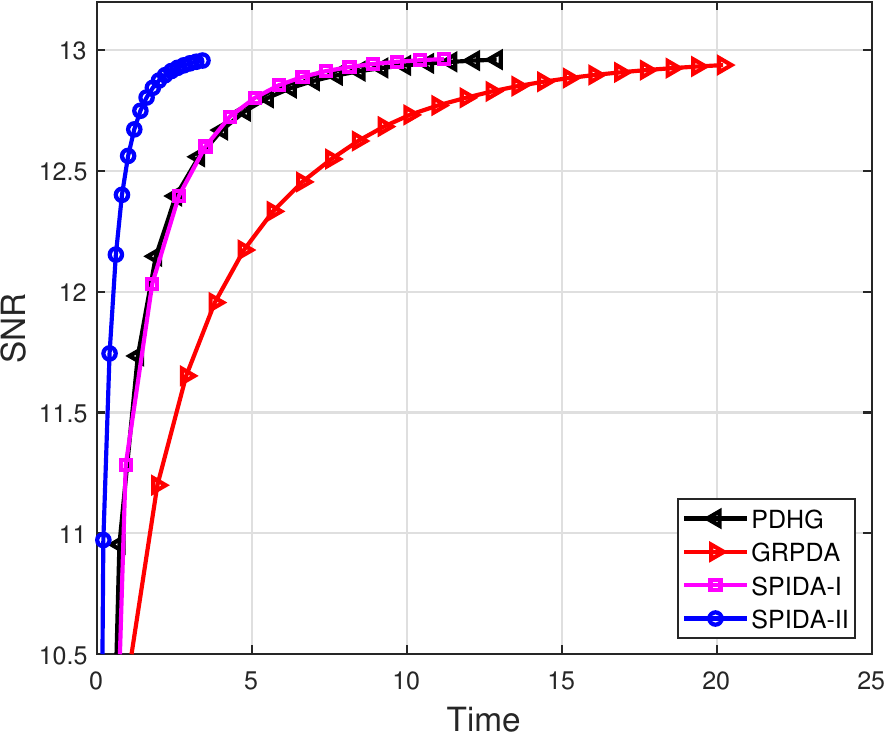}}
	\subfigure[flinstones]{
		\includegraphics[width=0.485\textwidth,height=0.3\textheight]{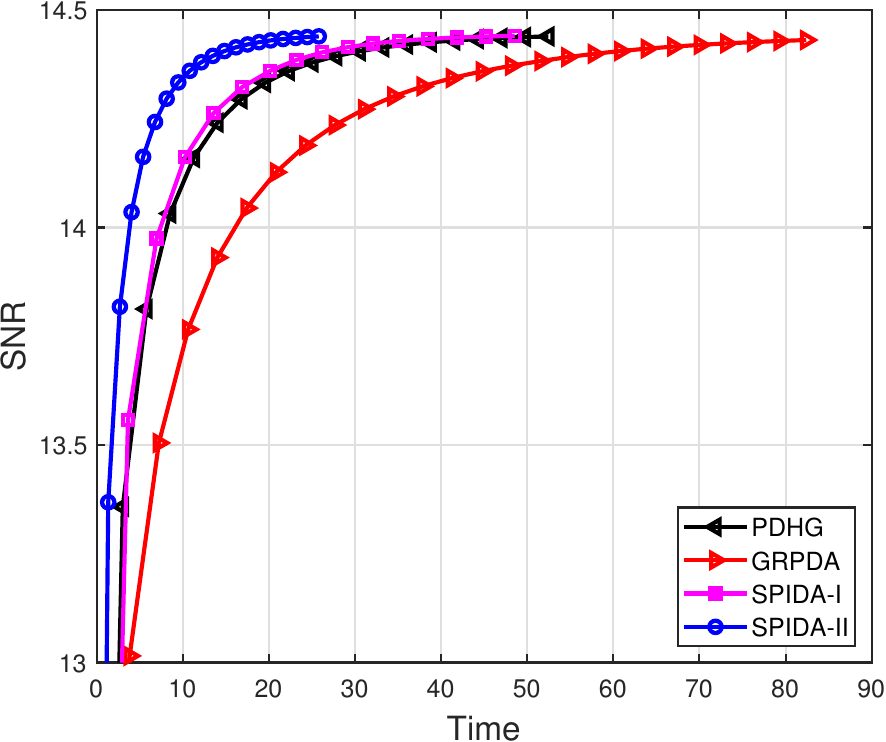}}
	\caption{Evolution of SNR value with respect to computing time for solving image restoration model \eqref{tv-saddle}.  }
	\label{fig7}
\end{figure}

\section{Conclusions}\label{Sec6}
In this paper, we proposed a new primal-dual algorithmic framework for convex-concave saddle point problems by applying the symmetric idea to compute the dual variable twice. Notice that a Bregman proximal regularization term is embedded in each subproblem, which is of benefit for us designing customized algorithms for some structured optimization problems. Moreover, we can gainfully understand the classical (linearized) augmented Lagrangian method and some parallel augmented Lagrangian-based splitting methods for linearly constrained convex optimization. A series of experiments demonstrate that our new algorithm works better than the other two compared methods as long as the dual subproblem (i.e., $y$-subproblem) is easy enough with cheap computational cost. In the future, we will consider some acceleration techniques such as extrapolation on the method. Besides, we notice that all subproblems of our algorithm are required to be solved exactly, which is expensive or impossible in some cases (e.g., see experiments in image restoration). Therefore, designing a practical inexact version of the proposed algorithm is also one of our future concerns.

\bigskip
\noindent {\bf Acknowledgments.} The authors are grateful to the anonymous referees for their close reading, insightful comments, and valuable suggestions, which greatly help us improve the quality of this paper. Moreover, the first author would like to thank Professor Min Yan for bringing his attention to \cite{CHZ16}. This research was supported in part by National Natural Science Foundation of China (Nos. 12371303 and 11901294), Zhejiang Provincial Natural Science Foundation of China (No. LZ24A010001), and Ningbo Natural Science Foundation (No. 2023J014).

\end{document}